\documentclass[11pt]{amsart}
\usepackage{t1enc}
\usepackage{amsthm,amssymb}
\usepackage{latexsym}
\usepackage{amsmath} 
\usepackage{graphics}
\usepackage{epsfig,multicol}
\usepackage{amsfonts}
\usepackage[latin1]{inputenc}
\usepackage[english]{babel}
\usepackage{ mathrsfs }
\usepackage{makecell}

\newtheorem{theorem}{Theorem}[subsection]
\newtheorem{proposition}{Proposition}[subsection]

\newtheorem{definition}{Definition}[subsection]

\newtheorem{remark}{Remark}[subsection]

\usepackage{hyperref}
\hypersetup{
    colorlinks=true,
    linkcolor=blue,
    filecolor=magenta,      
    urlcolor=cyan,
}

\def\hpic #1 #2 {\mbox{$\begin{array}[c]{l} \epsfig{file=#1,height=#2}
\end{array}$}}
 
\def\vpic #1 #2 {\mbox{$\begin{array}[c]{l} \epsfig{file=#1,width=#2}
\end{array}$}}

\newcommand{\D}{\displaystyle}

\newcommand {\5}{\vskip 5pt}
\begin{document}
\title{On the construction of knots and links from Thompson's groups.}
\author{Vaughan F. R. Jones}
\thanks{}

\begin{abstract} We review recent developments in the theory of Thompson group representations related to 
knot theory.
\end{abstract}
\maketitle
\tableofcontents
\section{Introduction} A few years ago now a method of geometric origin was introduced for constructing
representations of Thompson's groups $F$ and $T$ of piecewise linear homeomorphisms of $[0,1]$ and $S^1$
respectively. Applying this construction to the category of Conway tangles gave a way of constructing a link
from a Thompson group element. It was shown in \cite{jo4} that all links arise in this fashion. 
The knot theoretic outcome of this construction can be summed up in the following result, where the diagrams drawn
are hopefully sufficiently clear to indicate the general case (the dashed lines indicate what has been essentially
added to the pair of binary planar rooted trees-if one removes them from the knot on the right, the tree structures top and bottom 
should be apparent). For details consult \cite{jo4}.
\begin {theorem} Let $R$ be the ring of formal linear combinations (over $\mathbb Z$) of isotopy classes of unoriented link diagrams
with multiplication given by distant union and conjugation given by mirror image.
There is an $R$-module $\mathcal V$ with $R$-valued sesquilinear inner product $\langle,\rangle$, together with a privileged element
$\Omega \in \mathcal V$, and a $\langle,\rangle$-preserving $R$-linear action $\pi$ of Thompson's group $F$ on $\mathcal V$ such
that, for instance, $$\mbox{ for $g=$} \vpic{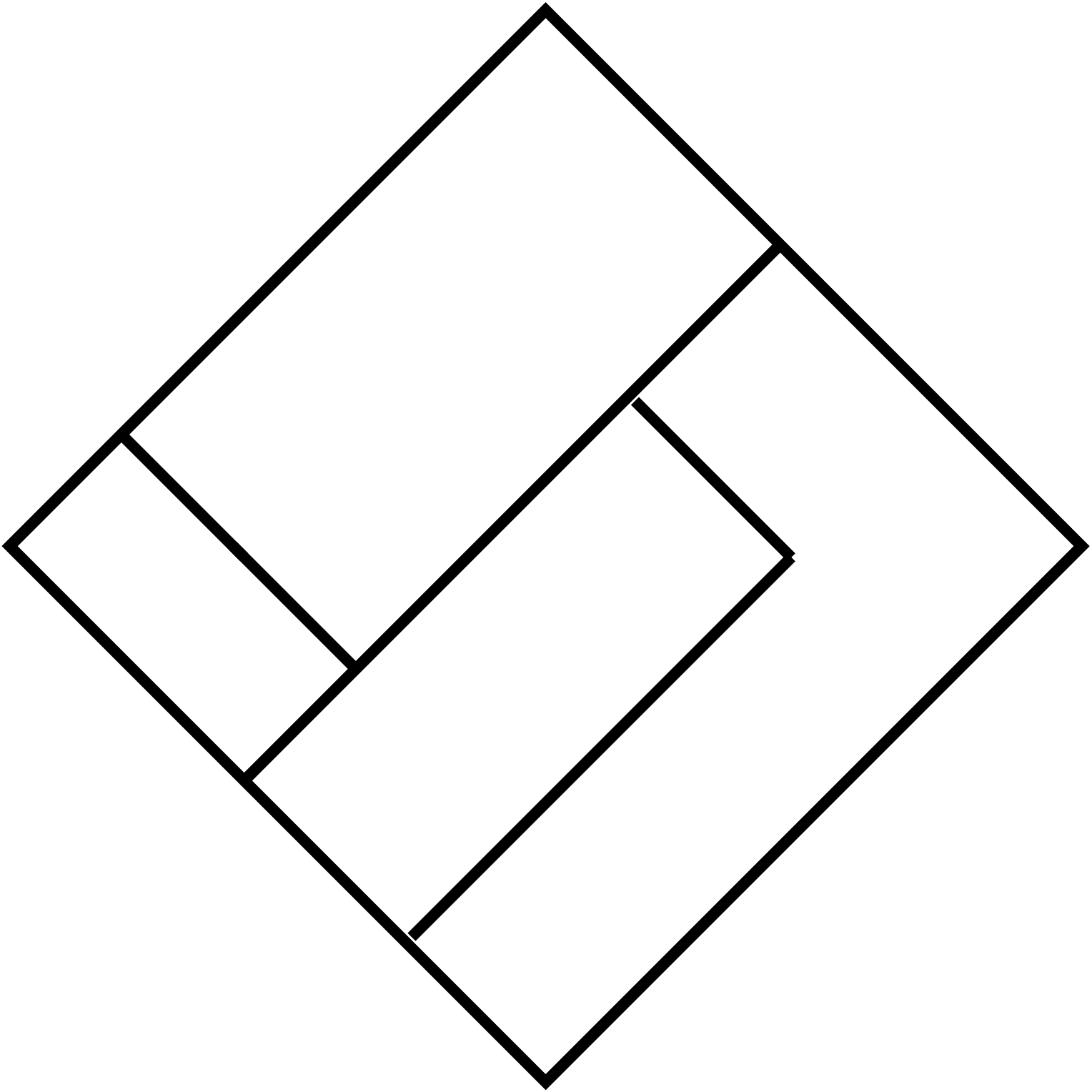} {1in}  ,\qquad \langle \pi(g) \Omega,\Omega \rangle= \vpic {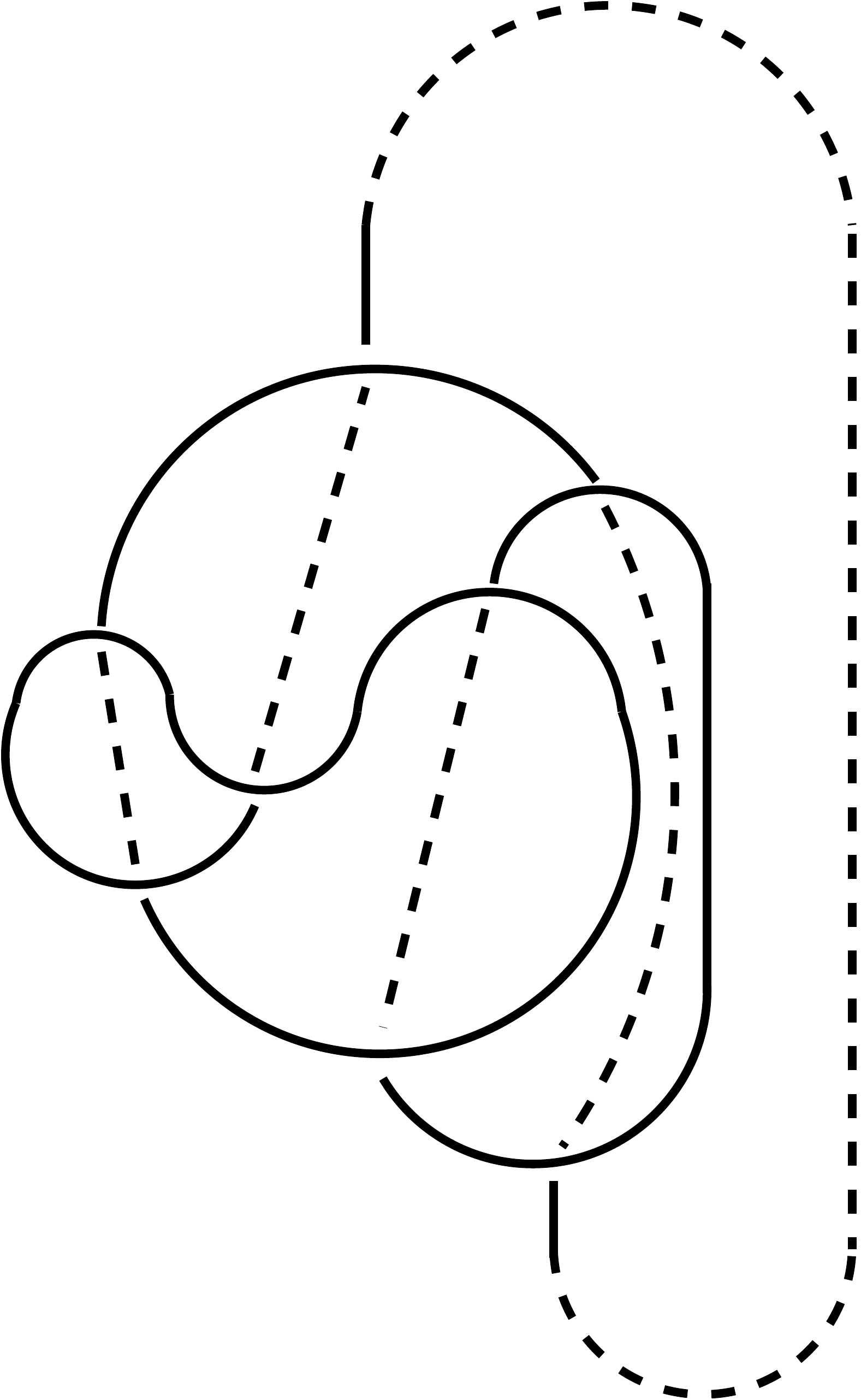} {1in} $$ 
if  $g\in F$ is given by the "pair of trees" representation (see \cite{CFP})
as above. Morevoer any link in $\mathbb R^3$  arises in this way.
\end{theorem}

Since then this construction has been better understood, considerably simplified and generalised, though admittedly at the
cost of geometric understanding.  In this largely expository paper we will first describe the new simplified version of the construction with a
few new examples of actions of Thompson groups. We will then explain the particular context that leads to the theorem above-further
simplified by the use of the Thompson groups $F_3$ and $T_3$ rather than $F_2$ and $T_2$. Finally we will list a few  obvious  questions that 
remain open at this stage. 

\section{The directed set/functor method.}\label{construction}
A planar $k$-forest is the isotopy class of a disjoint union of  planar rooted trees all of whose vertices are adjacent to 
$k+1$ edges, embedded
in $\mathbb R^2$ with roots lying on $(\mathbb R,0)$ and  leaves lying on $(\mathbb R,1)$. The isotopies preserve the strip $(\mathbb R,[0,1])$ but may act nontrivially on the boundary. Planar $k$-forests form a category
in the obvious way with objects being $\mathbb N$ whose elements are identified with isotopy classes of sets of points on a line and whose morphisms are the  planar $k$-forests themselves,  which can be composed by stacking a forest in 
$(\mathbb R,[0,1])$ on top of another, lining up the leaves of the one on the bottom with 
the roots of the other by isotopy then rescaling the $y$  axis to return
to a forest in  $(\mathbb R,[0,1])$. 

We will call this category $\mathcal F_k$.

The set of 
morphisms from $1$ to $n$ in $\mathcal F_k$ is the set  of $k$-ary planar rooted trees $\mathfrak T_k$ and is a \emph{directed set} with $s\leq t$ iff there is and $f\in \mathcal F$ with $t=fs$.

It is useful to know the number of $k$-ary planar rooted trees.

\begin{proposition} \label{fc1} There are $FC(k,n)$ $k$-ary planar rooted trees with $n$ vertices where $FC(k,n)$ is the
Fuss-Catlan number $\D \frac{1}{(k-1)n+1}\binom{kn}{n}$.

\end{proposition}
\begin{proof} By attaching $k$ new trees to the root vertex we see that the number of  $k$-ary planar rooted trees with $n$ vertices satisfies the same recursion relation:
 $$ FC(k,n+1)=\sum _{\ell_1,\ell_2,\cdots \ell_k,\sum \ell_i=n} \prod _{i=1}^n FC(k,\ell_i).$$
 See \cite{concrete} for details about the Fuss-Catalan numbers.
\end{proof}
Given a functor $\Phi:\mathcal  F\rightarrow \mathcal C$ to a category $\mathcal C$ whose objects are sets, 
we define the direct system $S_\Phi$ which associates to each $t  \in \mathfrak T$, $t:1\rightarrow n$,
 the set $\Phi(target(t))=\Phi(n)$. 
 For
each $s\leq t$ we need to give $\iota_s^t$. For this observe 
that there is an $f\in \mathcal F$ for which $t=fs$ so we
define
 $$\iota_s^t =\Phi(f)$$
 which is an element of $Mor_{\mathcal C}(\Phi(target(s)),\Phi(target(t)))$ as required. The $\iota_s^t$ trivially
 satisfy the axioms of a direct system.
 
  As a slight variation on this theme, given a functor $\Phi:\mathcal  F\rightarrow \mathcal C$ to \underline{any} category $\mathcal C$, 
and an object $\omega \in \mathcal C$, form the category $\mathcal C^\omega$ whose objects are the sets $Mor_{\mathcal C}(\omega,obj)$ for every 
object $obj$ in $\mathcal C$, and
whose morphisms are composition with those of $\mathcal C$. The definition of the functor $\Phi^\omega:\mathcal F\rightarrow \mathcal C^\omega$
is obvious.
  Thus the direct system $S_{\Phi^\omega}$ associates to each $t  \in \mathfrak T$, $t:1\rightarrow n$,
 the set $Mor_{\mathcal C}(\omega,\Phi(n))$. Given $s\leq t$ let  $f\in \mathcal F$ be such that $t=fs$.Then
for $\kappa \in Mor_{\mathcal C}(\omega,\Phi(target(s)))$,
 $$\iota_s^t(\kappa) =\Phi(f)\circ \kappa$$
 which is an element of $Mor_{\mathcal C}(\omega,\Phi(target(t)))$. 
 
 As in \cite{nogo} we consider the direct limit:
$$ \underset{\rightarrow} \lim S_\Phi=\{(t, x) \mbox{ with } t \in  \mathfrak T, x\in \Phi(target(t))\} / \sim$$ 
where $(t,x)\sim (s,y)$ iff there are $r\in \mathfrak T, z\in \Phi(target(z))$ with $t=fr, s=gr$ and $\Phi(f)(x)=z=\Phi(g)(y)$.
\vskip 5pt
\emph{We use $\displaystyle {t \over x}$ to denote the equivalence class of $(t,x)$ mod $\sim$.}
\vskip 5pt
The limit $ \underset{\rightarrow} \lim S_\Phi$ will inherit structure from
the category $\mathcal C$. For instance if the objects of $\mathcal C$ are Hilbert spaces and the morphisms are isometries then 
$ \underset{\rightarrow} \lim S_\Phi$ will be a pre-Hilbert space which may be completed to a Hilbert space which we will also call the direct
limit unless special care is required.

As was observed in \cite{nogo}, if we let 
$\Phi$ be the identity functor and choose $\omega$ to be the tree with one leaf,
then  the inductive limit consists of equivalence classes of pairs $\D \frac{t}{x}$ where $t\in \mathcal T$ and $x\in \Phi(target(t))=Mor(1,target(t))$.
But $Mor(1,target(t))$ is nothing but $s\in \mathcal T$ with $target(s)=target(t)$, i.e. trees with the same number of leaves as $t$.
Thus the inductive limit is nothing but the (Brown-)Thompson group $F_k$ with group law $${r\over s}{s\over t}={r\over t}.$$ 

For instance the following tree fraction gives an element of $F_3$:

$$ \vpic {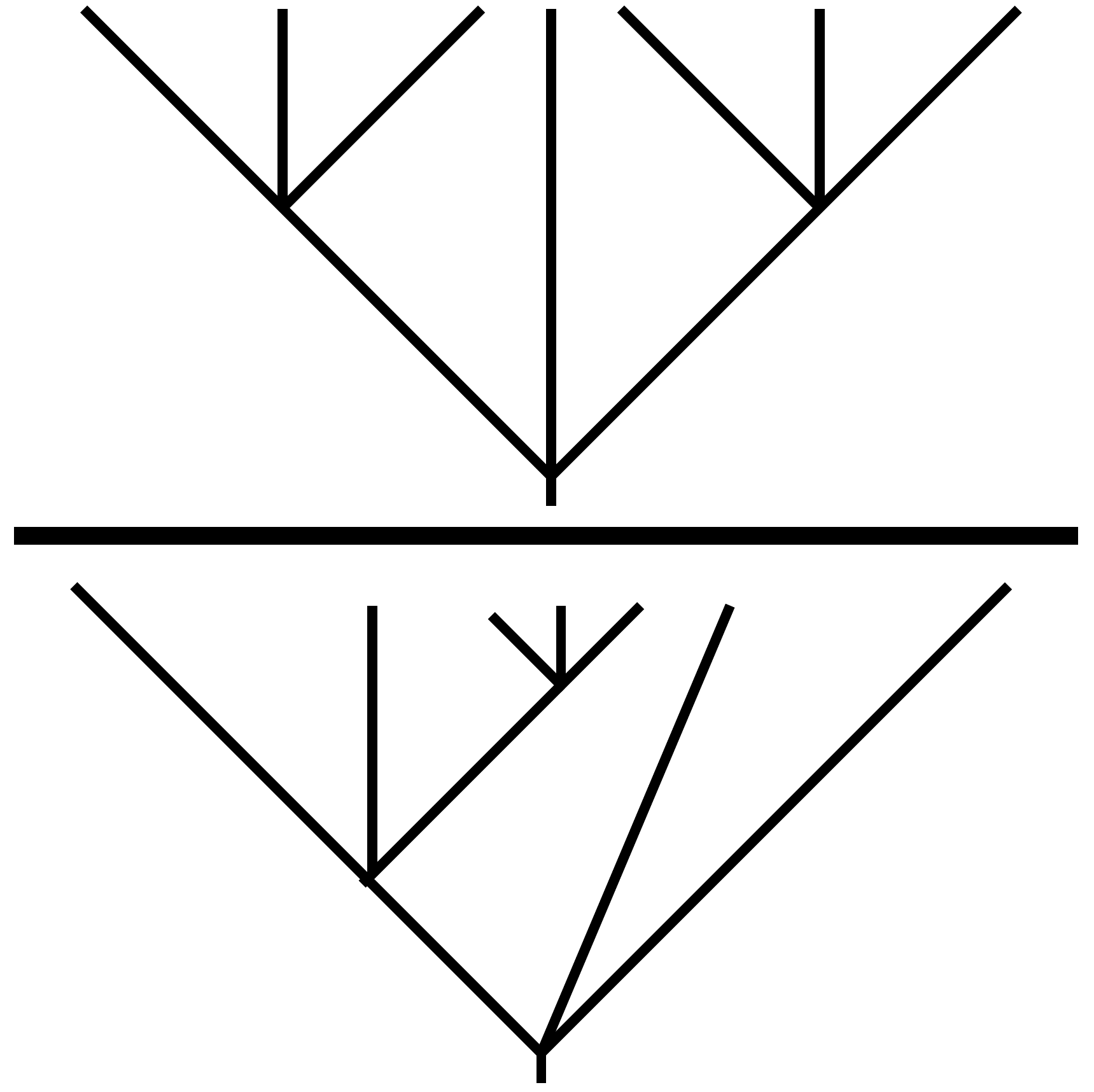} {1in}  $$
It is part of the philosophy of this paper to think of such an element also as the diagram obtained by flipping the numerator 
upside down and attaching its leaves to those of the denominator so that, for instance, the above element can equally be 
written $$ \vpic {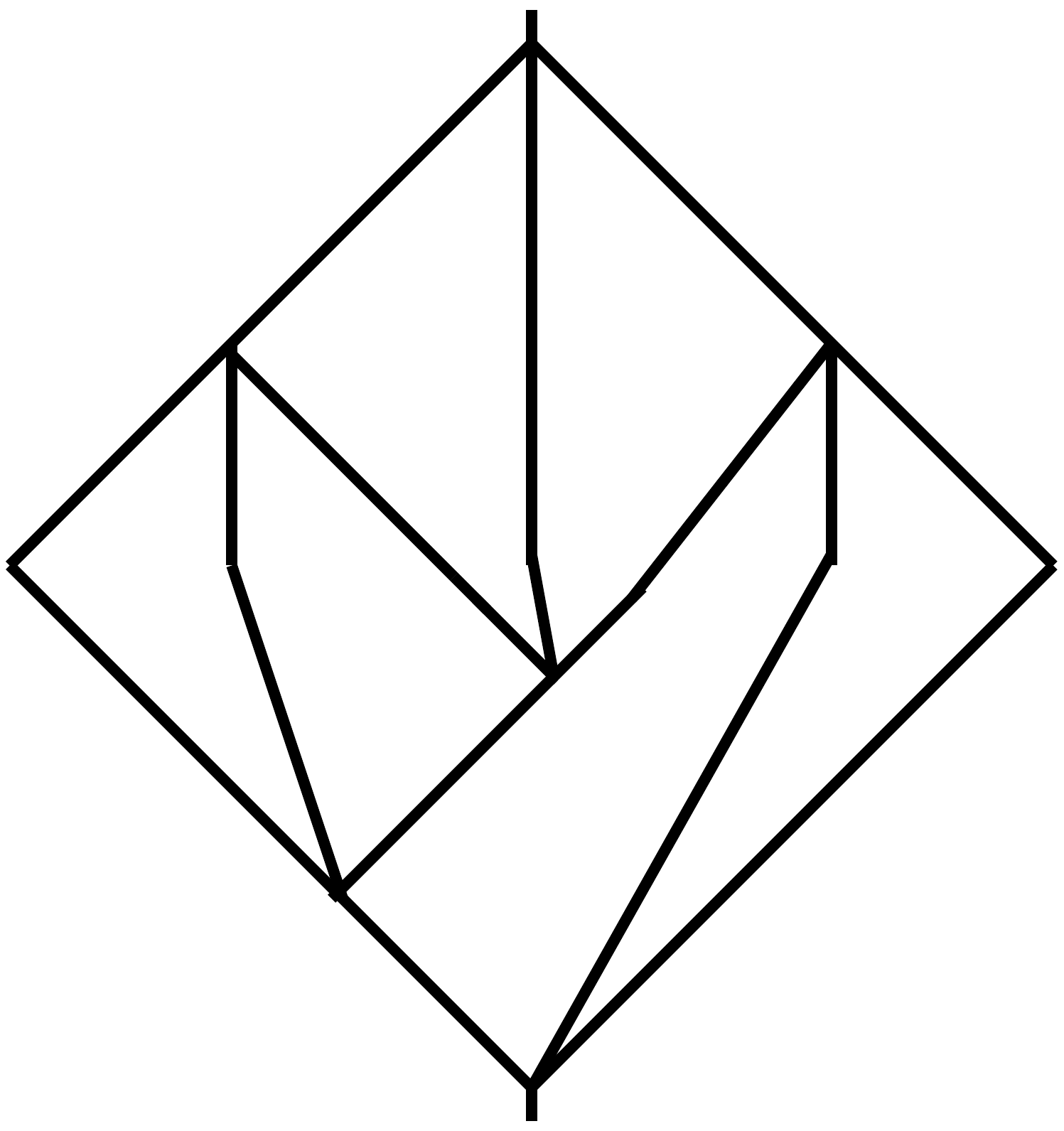} {1in}  $$

(Geometrically this makes sense-$F_3$ is a group of piecewise linear homeomorphisms constructed from intervals which are
sent by scaling transformations to other intervals. If we associate an interval with each leaf of a tree by the rule that
each vertex of the tree splits an interval up into three adjacent intervals of equal width, then the element of $F_3$ given
by a diagram as above maps an interval of a leaf of the bottom tree to the interval of the leaf of the top tree to which it
is attached.)

 With this definition of the group we may construct actions in a simple way. For any functor $\Phi$,
$ \underset{\rightarrow} \lim S_\Phi$ carries a natural action of $F_k$ 
defined as follows:
$$\frac{s}{t}(\frac{t}{x})=\frac{s}{x}$$ where $s,t\in \mathfrak T_k$ with $target(s)=target(t)=n$ and $x\in \Phi(n)$.  A Thompson group element given as a pair of trees with $m$ leaves, and an element of
 $ \underset{\rightarrow} \lim S_\Phi$ given as a pair (tree with $n$ leaves,element of $\Phi(n)$), may not be immediately composable by the above forumula, but they can always be ``stabilised'' to be so within their equivalence classes. 
 
 The Thompson group action preserves the structure of 
 $ \underset{\rightarrow} \lim S_\Phi$ so for instance in the Hilbert space case the representations are unitary.

 \section{The connection with knots.}\label{knots}
 
 A Conway tangle is an isotopy class of rectangles with $m$ ``top'' and $n$ ``bottom'' boundary points, containing smooth curves called
 strings with under and over crossings which meet the boundary transversally in the $m+n$ boundary points. The isotopies are 
 considered to contain the three Reidemeister moves but must fix the top and bottom edges of the rectangles.
 Conway tangles form a category whose objects are $0\cup \mathbb N$ with the non-negative integer $n$ being identified with
 isotopy classes of sets of $n$ points in an interval. The morphisms are then the Conway tangles with the obvious stacking 
 as composition: 
 
 $$\mbox{A morphism in $\mathfrak C$ from 3 to 5}: \qquad \vpic {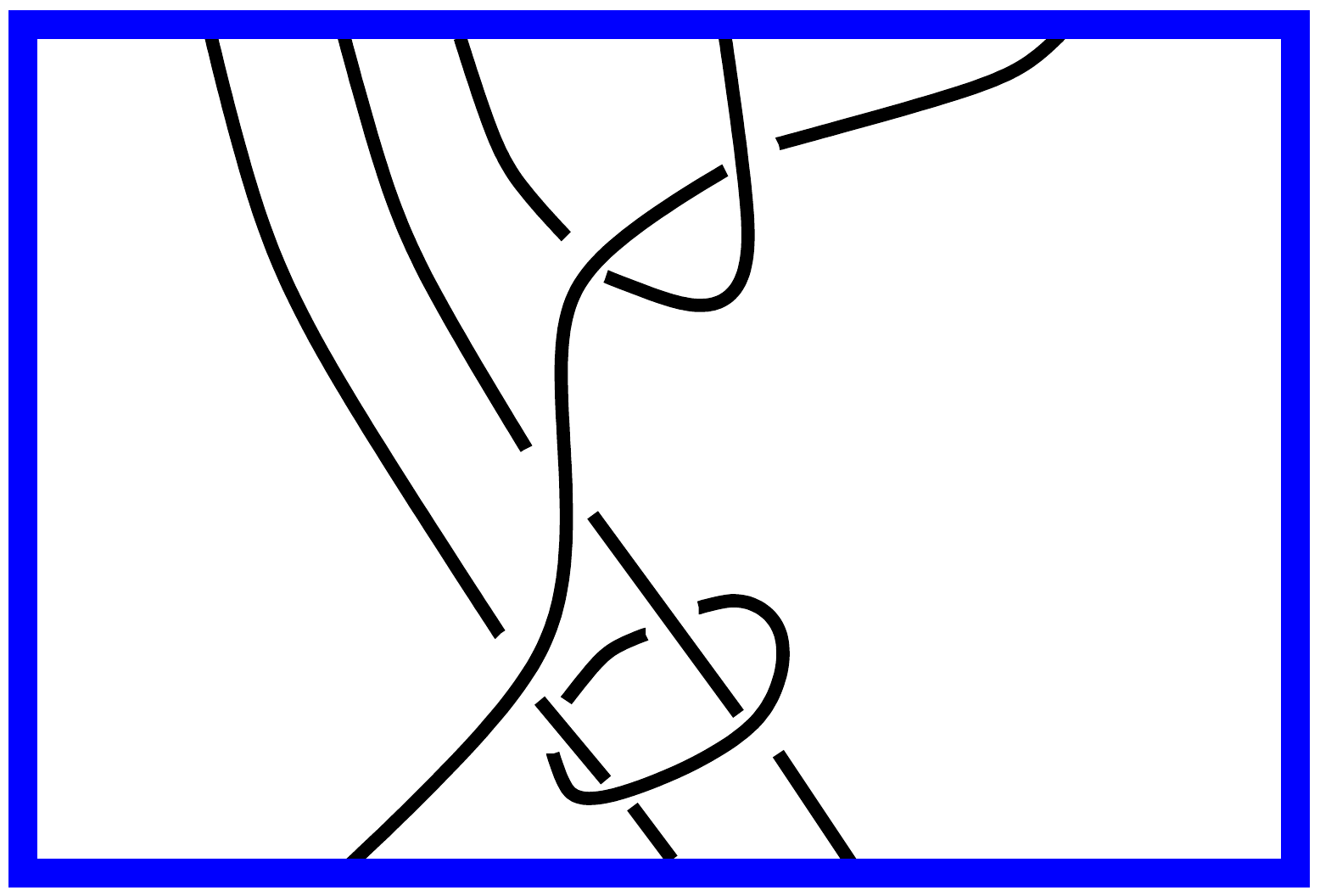} {2in} $$
 
 Of course the set of morphisms from $m$ to $n$ is empty if $m+n$ is odd.
 
 \begin{definition}
 The Conway tangles defined above will be called the category of tangles $\mathfrak C$.
 \end{definition}
 
 We will apply the construction of the previous section to the ternary Thompson group $F_3$ to obtain actions of it on
 spaces of tangles. We distinguish three slightly different ways to do this.
 \begin{enumerate}

\item  Set theoretic version: 

To perform the construction of the previous section we need to define a functor from ternary forests to $\mathfrak C$. 
\begin{definition} The functor $\Phi: \mathcal F_3\rightarrow \mathfrak C$ is defined as follows:
\5
a) On objects  $\Phi (n)=n$ so that the roots of a planar forest  are sent to the boundary points at the bottom of a tangle, from
left to right.
\5
b) On morphisms (i.e. forests), $\Phi (f)$ is defined to be the tangle obtained by isotoping the forest to be in a rectangle with roots on 
the bottom edge and leaves on the top edge,  and  replacing each vertex of the forest with a crossing thus:

$$ \vpic {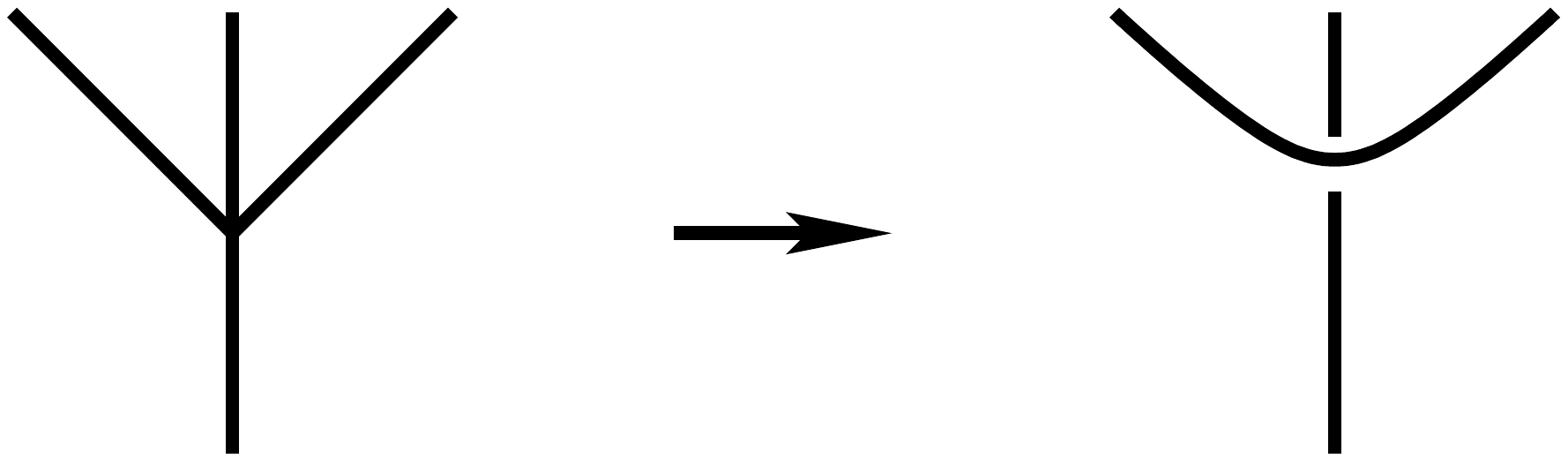}  {2in} $$

\end{definition}

Thus for instance if $f$ is the forest \vpic {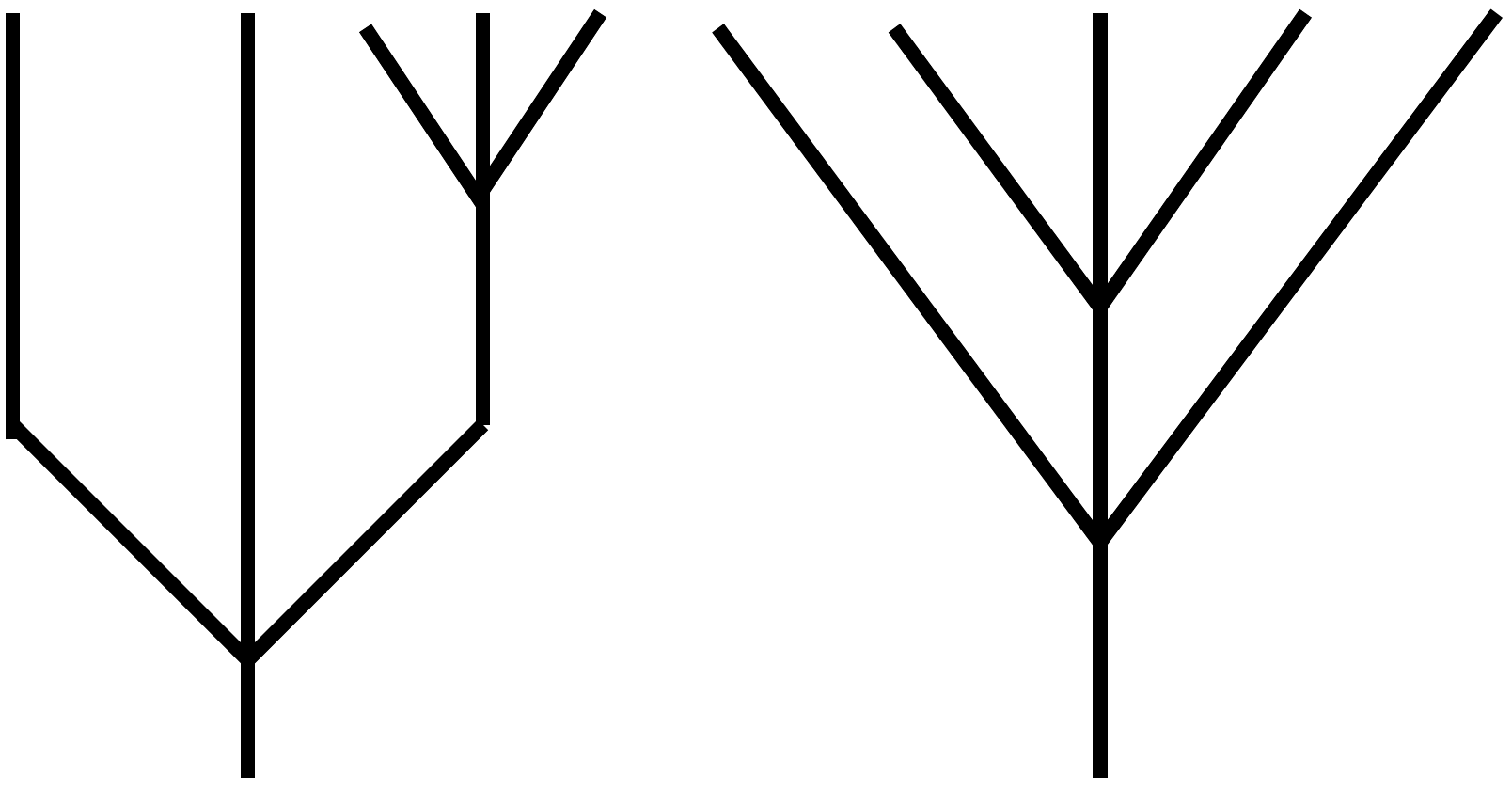} {1.5in}  then $$\Phi(f)= \vpic {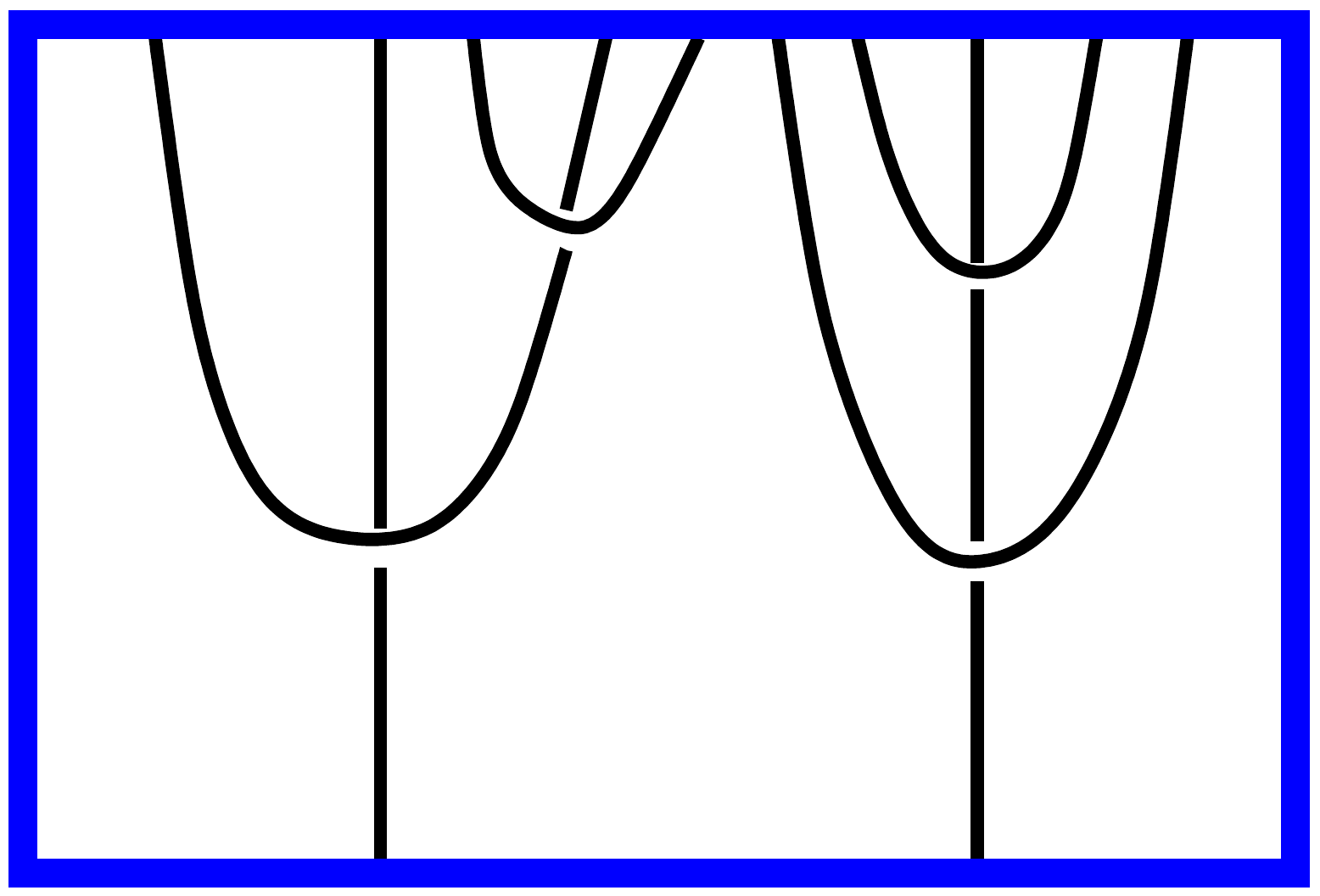} {1.5in} $$

The well definedness and functoriality of $\Phi$ are obvious. 
\5 
By the machinery of the previous section we obtain an action of $F_3$ on a set $\tilde {\mathfrak C}$, the direct limit of sets of tangles.
An element of the set $\tilde {\mathfrak C}$ is the equivalence class of a pair $(t,T)$ where $t$ is a ternary planar rooted tree with $n$
leaves and $T$ is a $(1,n)$ Conway tangle. Adding a single vertex to $t$ corresponds to adding a single crossing to $T$ and this
generates the equivalence relation. Thus an element of $\tilde {\mathfrak C}$ can be thought of as an infinite tangle with one
boundary point at the bottom and eventually ending up with a lot of simple crossings as below:

$$\vpic {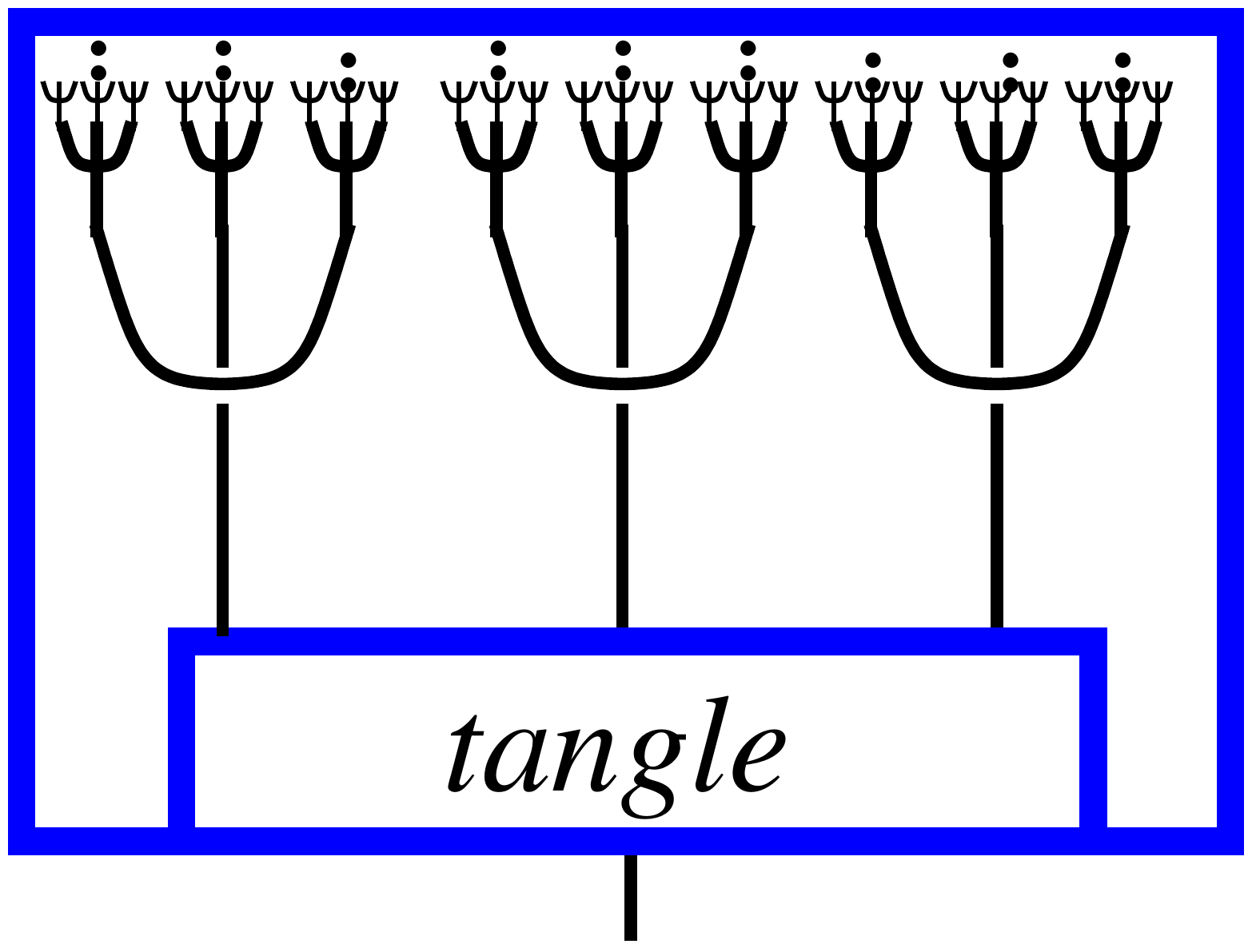} {3in} $$

The original geometric intuition of our construction was to think of a Thompson group element as giving a piecewise linear
foliation of a rectangle attaching points at the bottom to their images on the top, and stacking it on top of the above picture to
give a new such picture. See \cite{jo2}.
\item Linearised version:
\5
Recall that  $R$ is  the ring of formal linear combinations (over $\mathbb Z$) of (three dimensional) isotopy classes of unoriented links with distant union as multiplication.
An unoriented link acts on a tangle simply by inserting a diagram for it in any region of the tangle. 

One may alter the  above construction by replacing the set of $(1,2k+1)$ tangles by the free $R$-module 
$R\mathfrak C_{1,k}$ having
those tangles as a basis. This way the direct limit $V$ is also an $R$-module. More importantly mirror image defines
an involution on $R$ and there is a sesquilinear form $\langle S,T\rangle$ on each $R\mathfrak C_{1,k}$ obtained by relecting 
the tangle $T$ about the top side of its rectangle, then placing it above $S$ and connecting all the boundary points in
the obvious way to obtain an element of $R$. 

Unfortunately the connecting maps $\iota_s^t$ of the direct system do not preserve $\langle,\rangle$. But this is easily
remedied by adjoining a formal variable $\sqrt \delta$ and its inverse to $R$ to obtain $R[\displaystyle \sqrt \delta,\frac{1}{\sqrt \delta}]$. One then modifies the functor $\Phi$ by multiplying the $R[\displaystyle \sqrt \delta,\frac{1}{\sqrt \delta}]$-linear map induced by $\Phi$ in (i) above (by its action on a the basis of tangles)
by $(\frac{1}{\sqrt \delta})^p$ where $p$ is the number of vertices in the forest. 
Then the connecting maps preserve the sesquilinear form which thus passes to a sesquilinear form on the direct limit
$R[\displaystyle \sqrt \delta,\frac{1}{\sqrt \delta}]$-module which is tautologically preserved by the action of $F_3$. To 
simplify notation we will continue to use $R$ for $R[\displaystyle \sqrt \delta,\frac{1}{\sqrt \delta}]$.

Now we are finally at the interesting bit.  Given a  representation of a group $G$ ,on an $R$-module $V$, $g\mapsto u_g$, preserving a sesquilinear form $\langle,\rangle$, the \emph{coefficients} of the representation are the functions $$g\mapsto \langle u_g(\xi),\eta\rangle$$ as $\xi$ and$\eta$ vary in $V$.  But our construction of the direct limit gives us a privileged
vector in $V$, namely the equivalence class of the $(1,1)$ tangle $\omega$ consisting of a single straight string connecting the boundary points of a rectangle
with two boundary points altogether. We will call this vector $\Omega$. If we want to think of $\Omega$ as an element
of $V$ thought of as an infinite tangle, it is just: $$ \Omega = \vpic {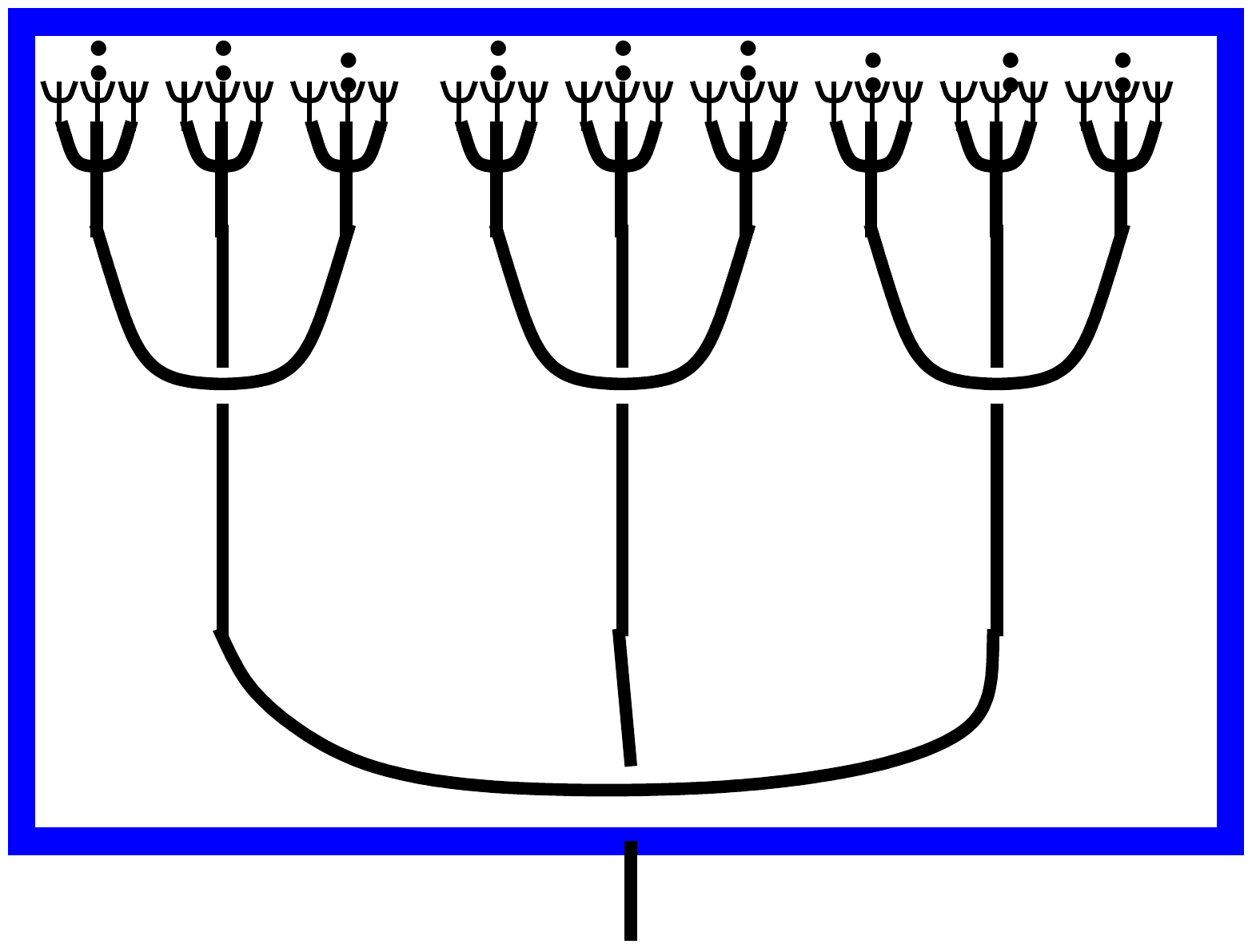} {2in} $$ (normalised by the appropriate 
power of $\delta$ for a given finite approximation).

Since $\Omega$ is a privileged vector, we would like to know the function on $F_3$ given by the coefficient $\langle u_g(\Omega),\Omega\rangle$,
$g$ being an element of $F_3$ and $u_g$ being the representation we have constructed. It is tempting to call the 
vector $\Omega$ the ``vacuum vector'' so that by analogy with physics (strengthened by the next section on
topological quantum field theory) we offer the following:

\begin{definition}
The element $\langle u_g(\Omega),\Omega\rangle$ of $R[\displaystyle \sqrt \delta,\frac{1}{\sqrt \delta}]$ is called 
the \emph{vacuum expectation value} of $g\in F_3$. (It is just a power of $\delta$ times a tangle.)
\end{definition}

It is not hard to calculate this element of $R$ if we follow the definitions carefully.

Let $\D g=\frac{s}{t}$ be an element of $F_3$ where $s$ and $t$ are planar rooted ternary trees with the same number of leaves.
$\Omega \in V$ is given by $\frac{1}{\omega}$ where $1$ is the tree with no vertices.
To calculate $u_g(\Omega)$ we need to stabilise $1$ so that we can apply the formula defining the representation.
Thus we write $\frac{1}{\omega}=\frac{t}{\Phi(t)}$ (recall that $\Phi$ is defined on a tree by changing all the vertices to crossings). Thus by definitino $$u_g(\Omega)=\frac{s}{\Phi(t)}.$$
To evaluate the sesquilinear form we must write $\Omega$ in the form $\D \frac{s} {something}$ and clearly that something
is $\Phi(s)$. Thus the coefficient is obtained by attaching $s$ to an upside down copy of $t$,  joining the top vertex to
the bottom one and replacing vertices by crossings, thus:
\5\5
If $s=\vpic {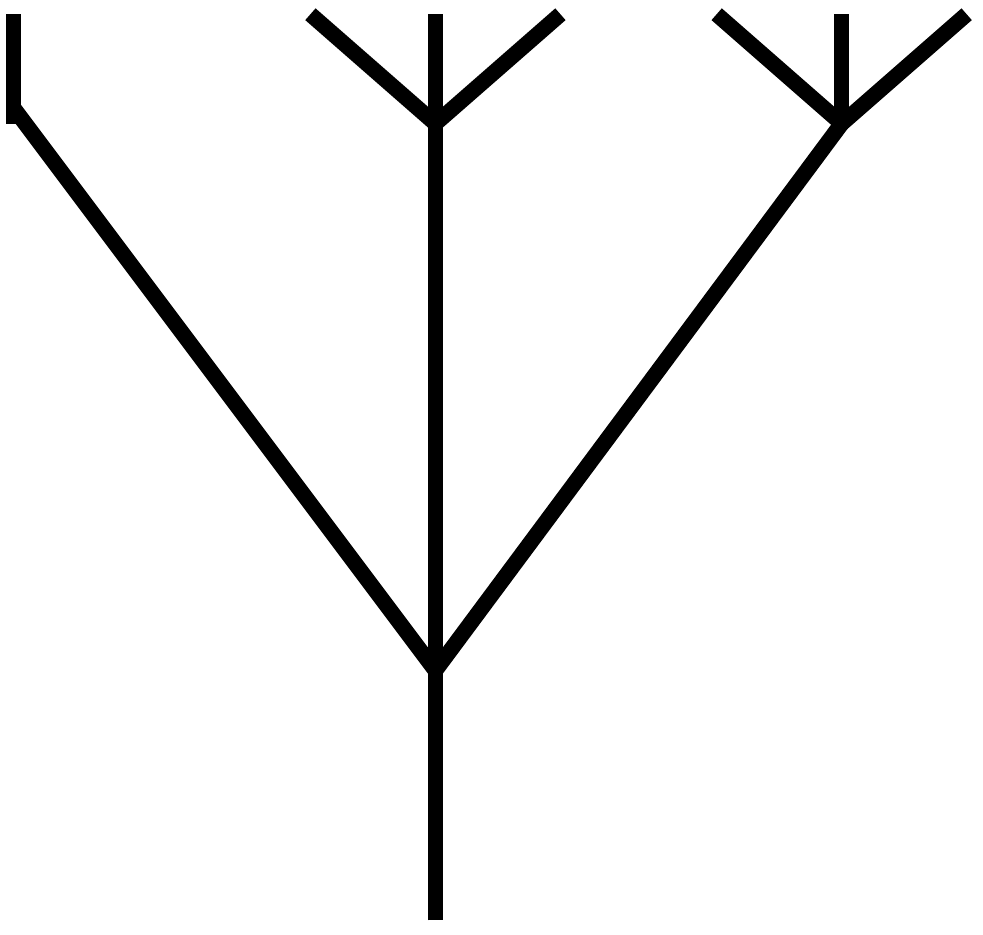} {0.6in} $ and $t=\vpic {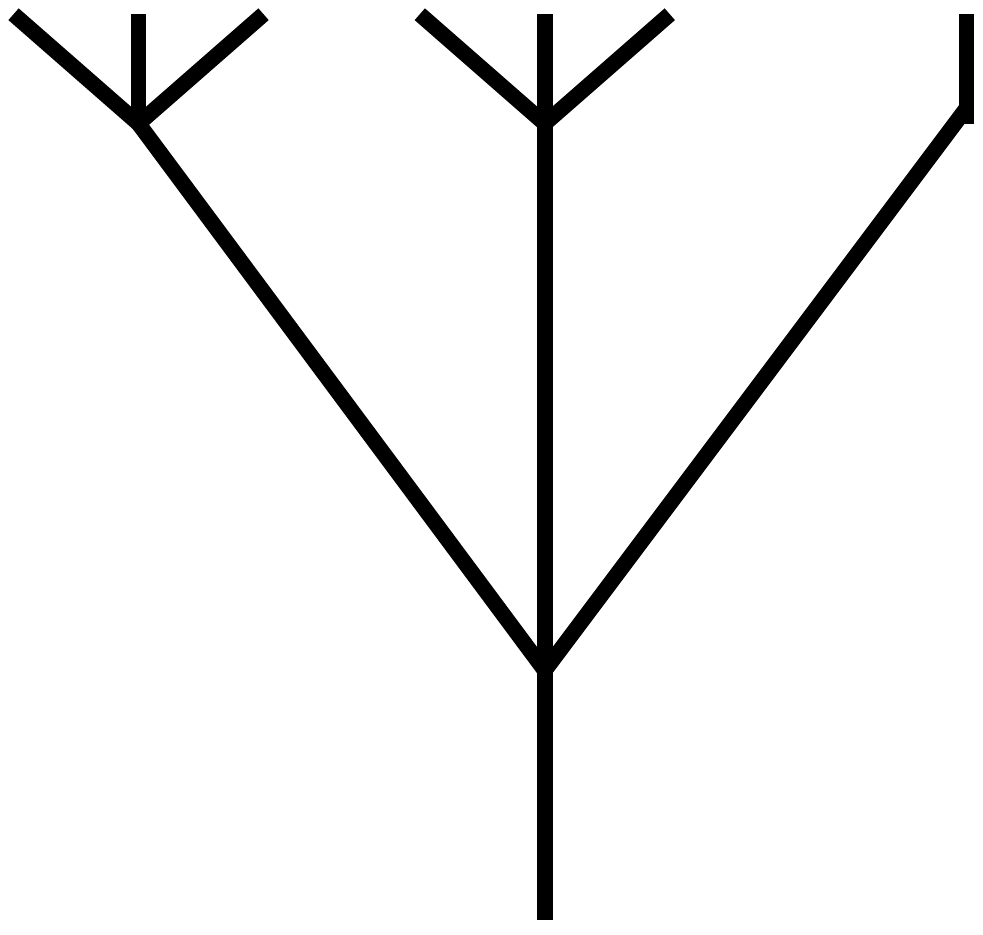} {0.8in} $ then
 $$\langle u_g \Omega, \Omega\rangle = \frac {1}{\delta^2}\hspace{10pt} \vpic {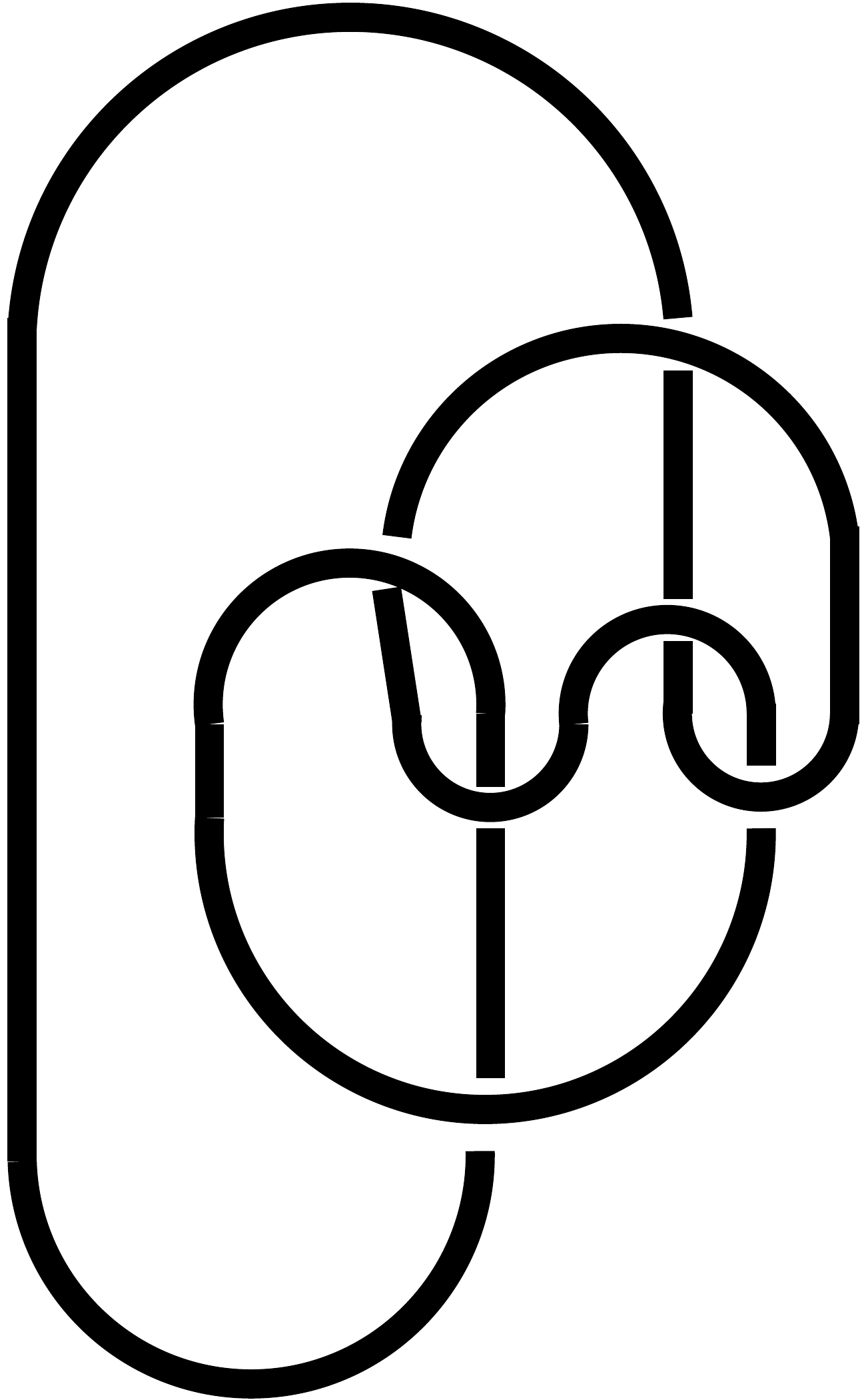} {1in} $$
 
 The factor $\frac{1}{\delta^2}$ comes innocently from the normalisation of the functor $\Phi$.
 The picture is fairly obviously a trefoil. 
 
 \begin{definition} \label{linkfromf3} If $g\in F_3$ we call $L(g)$ the link $\delta^n \langle u_g \Omega, \Omega\rangle$ for 
 the unique representation of $g$ as a pair of trees with a minimal number of vertices.
 \end{definition}
 Then we have:
 \begin{theorem}\label{alexander}
  
 Any knot or link can be obtained as $L(g)$ for some
 $g\in F_3$.
 \end{theorem}
 
 
 These vacuum expectation values are inherently unoriented.  There are two ways to handle oriented links, the most
 powerful of which is presented in \cite{homflypt}. But the easiest way is to use the following.
 \begin{definition} \label{oriented}
 
Let  $\overset{\rightarrow}{F_3}<F_3$ be the subgroup of elements
for whose pair of trees presentation the chequerboard shading gives a Seifert surface.
\end{definition}

For $F_2$ this subgroup was identified in \cite{gs} as being isomorphic to $F_3$! See also \cite{Ren}.

\item Skein theory version. \label{skein}

In the simple linearised version one may easily specialise $\delta$ ($\neq 0$) and use the complex numbers as
coefficients. But each approximating space to the inductive limit is infinite dimensional. This can be remedied 
by taking a skein theory relation (\cite{conway},\cite{Kff2}) and applying it to the approximating vector spaces spanned by tangles.
This is entirely compatible with the Thompson group action. The vacuum expectation value will then be just
the link invariant of the skein theory for the link $L(g)$ of definition \ref{linkfromf3}. Since we are dealing with
an unoriented theory the skein theories will have to be unoriented also and we will have to play the usual
regular isotopy game. Indeed the vacuum expectation value will be an invariant of regular isotopy if we use
the Kauffman bracket or the Kauffman polynomial. Moreover the diagram for $L(g)$ can be considered up to
regular isotopy. Since the proof below of the realisation of all links as $L(g)$ actually uses a lot of type I
Reidemeister moves, one may ask whether all \emph{regular isotopy classes} of link diagrams actually arise as
$L(g)$.
\item TQFT version:  We may ``apply a (unitary) TQFT'' at any stage in the above procedures, provided it is unoriented.
This means that the approximating subspaces for the direct limit are finite dimensional Hilbert spaces and the 
connecting maps $\iota_s^t$ are isometries so the direct limit vector space is a pre-Hilbert space on which the 
Thompson group acts by isometries so we can complete and obtain a \emph{unitary} representation of 
the Thompson group. 

The vacuum expectation values of the unitary representation can then always be calculated as statistical
mechanical sums as in \cite{J9}.
\end{enumerate}

\section{Relationship with the original construction- proof of theorem \ref{alexander}.}

It is possible to understand the construction of \cite{jo4}, which we gave in the introduction, in terms of a natural embedding $\phi$ of
$F_2$ in $F_3$. Take a binary rooted planar tree and simply attach another leave to the middle of each vertex thus:

$$g= \vpic {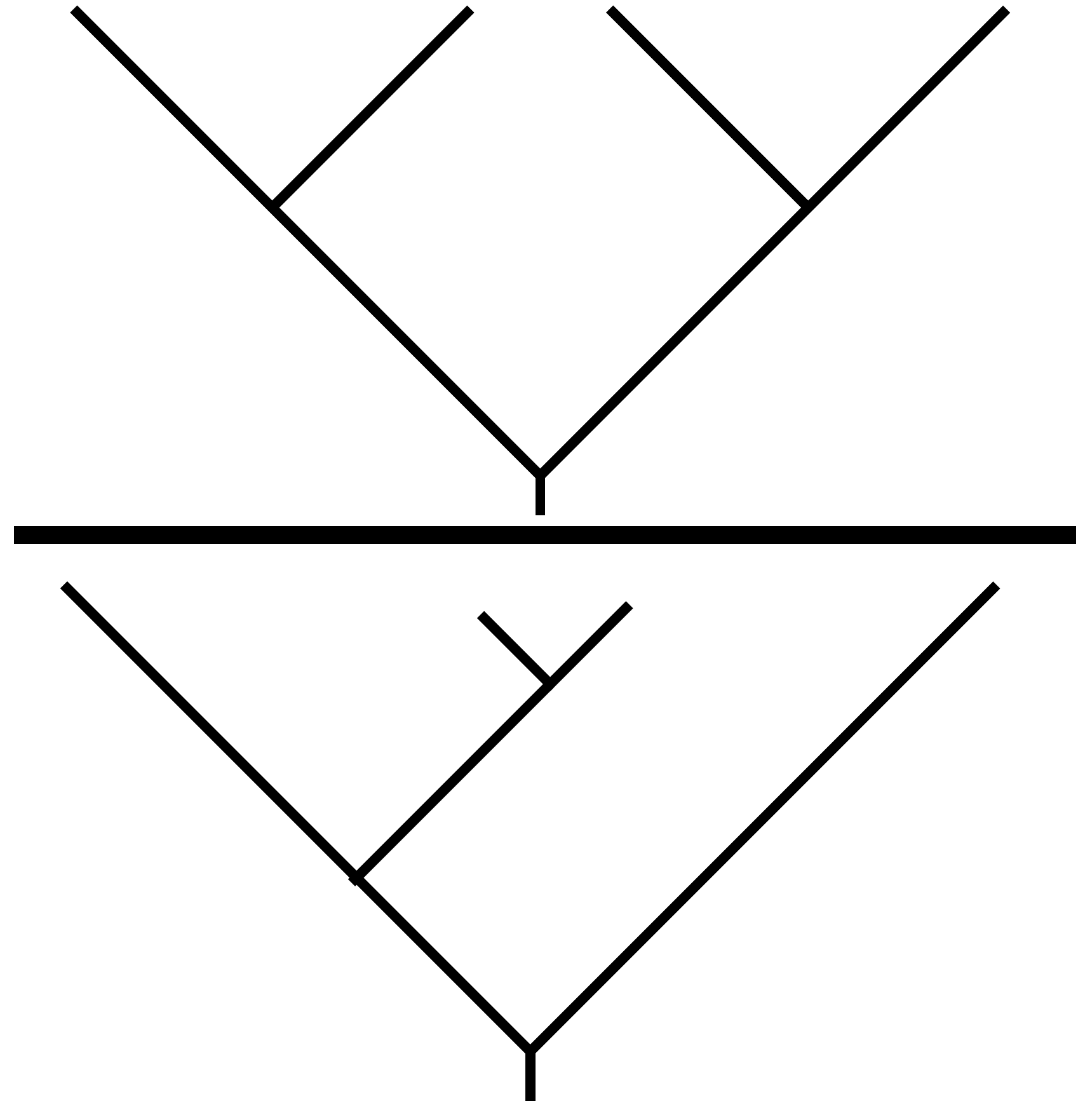} {1.5in}  \mapsto \phi(g)=\qquad \vpic {f3} {1.5in} $$

By construction $\langle u_{\phi(g)} \Omega,\Omega\rangle$ is the same as the coefficient of $g\in F_2$ defined in \cite{jo4}. When
there is no ambiguity we will identify $F_2$ with $\phi(F_2)$ as a subgroup of $F_3$.

Since $F_3$ is much bigger than $F_2$ it should be possible to find a  simpler proof of the ``Alexander'' theorem \ref{alexander} that all 
links can be obtained as vacuum expectation values for elements of $F_3$. We will see that, if we try to imitate the proof
of \cite{jo4} we run into a problem with signs so that the proof of the weaker theorem seems harder than that of the stronger one!
So we will sketch a slightly improved version of the proof of \cite{jo4}, pointing out the difference between the $F_2$ and $F_3$
cases.

Proof of theorem \ref{alexander}:
\begin{proof} Given a link diagram $L$ for an unoriented link $L$ we start by forming the edge-signed planar graph 
$\Gamma (L)$ given by a chequerboard shading of $L$
as usual, thus:

$$ L= \quad \vpic {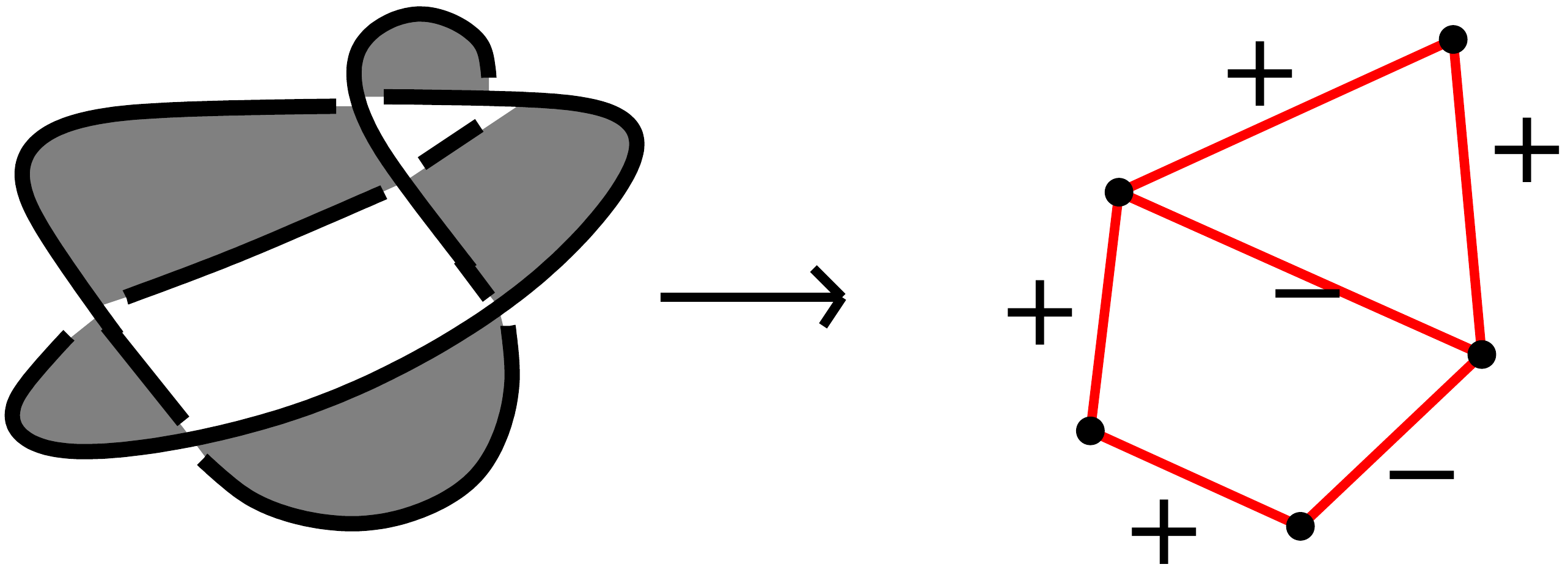} {3in}  \quad =\Gamma(L) $$ 

Where we have adopted the sign convention of \cite{jo4}:

$$  \vpic {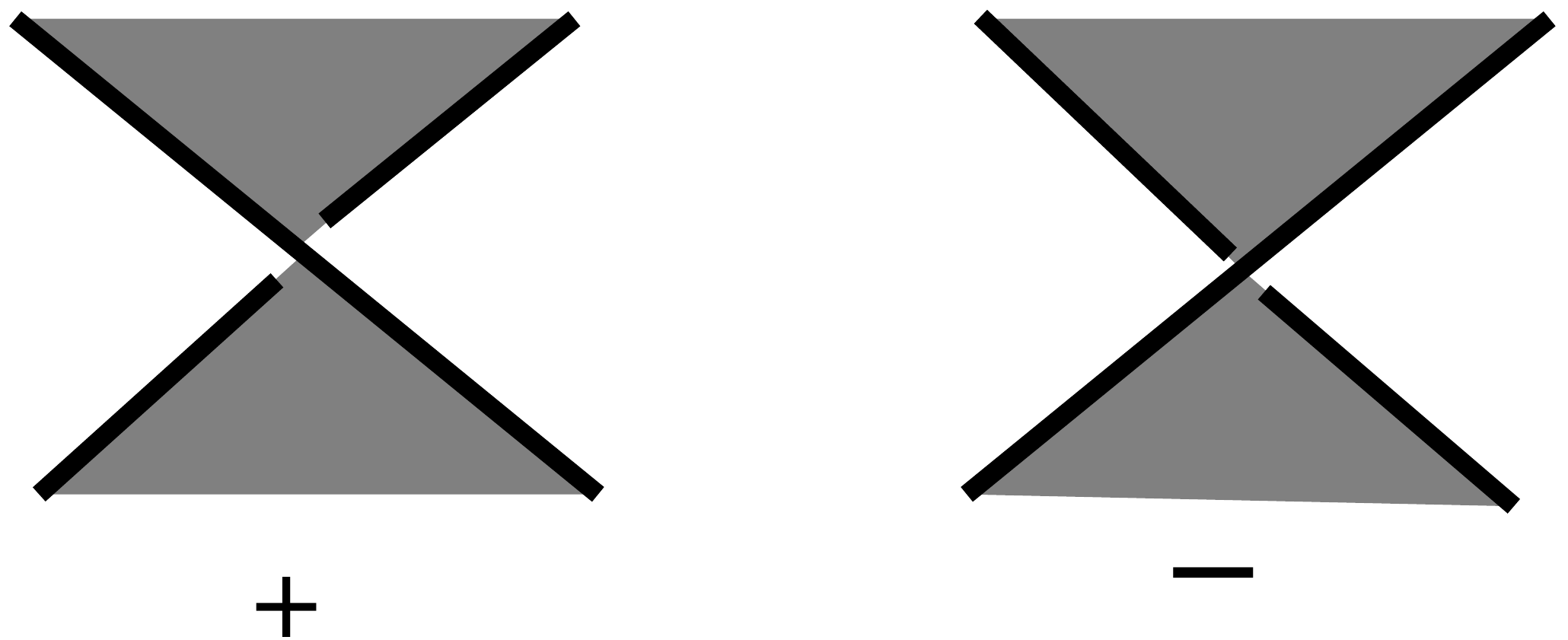} {2in}   $$ 
This process may be extended to Conway tangles, moving the vertices for boundary-touching faces to the boundary thus:

$$T= \quad \vpic {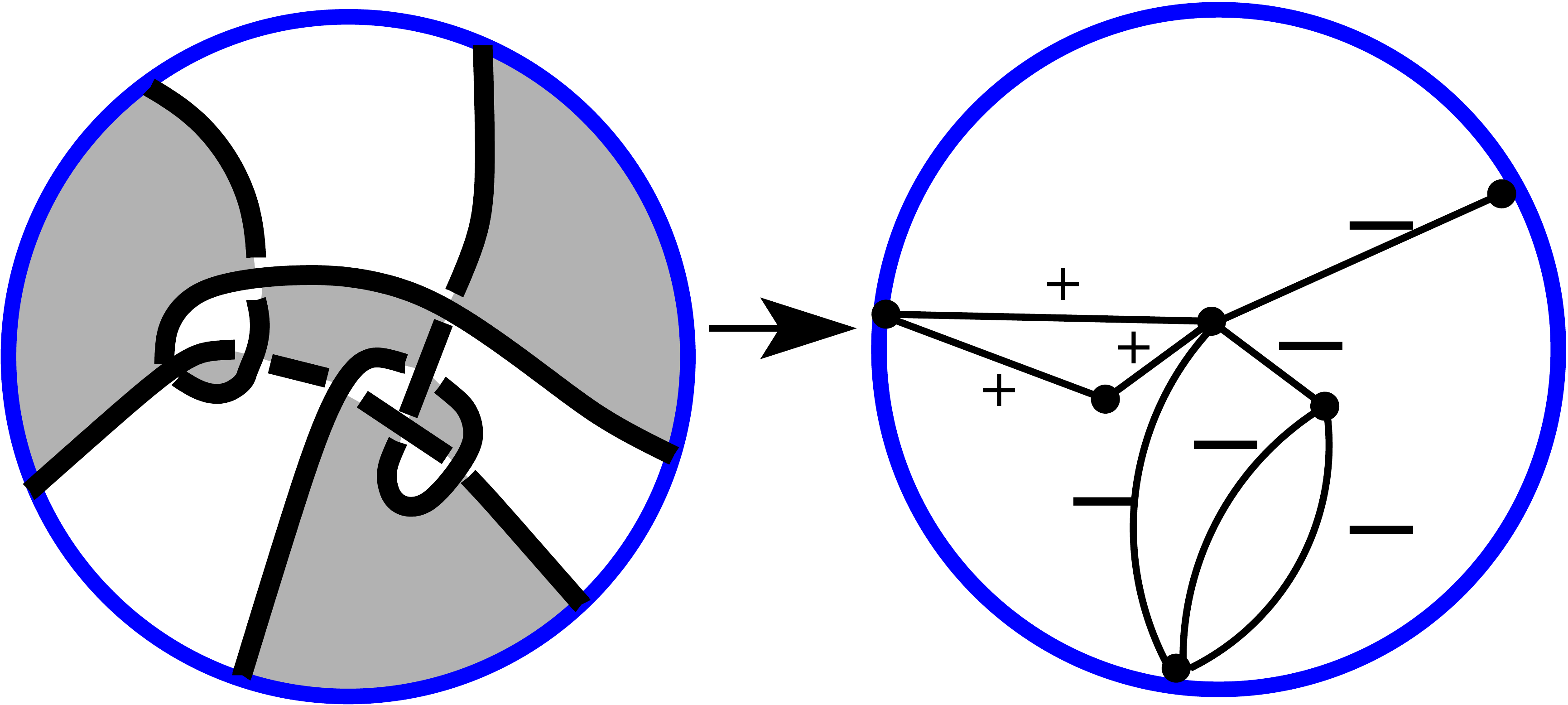} {3in}  =\Gamma(T) $$

It is important to note that this process is \emph{reversible}, a tangle can be obtained from any planar graph $\Gamma$
with some vertices on the boundary, by putting little crossings in the middle of the edges of $\Gamma$ and connecting them up around the faces of $\Gamma$ with the faces meeting the boundary having two points on the boundary rather than a
little crossing. (Or equivantely define the tangle as the intersection of some smooth disc with the link defined by $\Gamma$.)
This means that the map from Conway tangles to planar graphs in a disc is injective.

Observe that if the link diagram is of the form $L(g)$ (see \ref{linkfromf3}) then $\Gamma(L)$ has a special form, e.g.

$$ \vpic {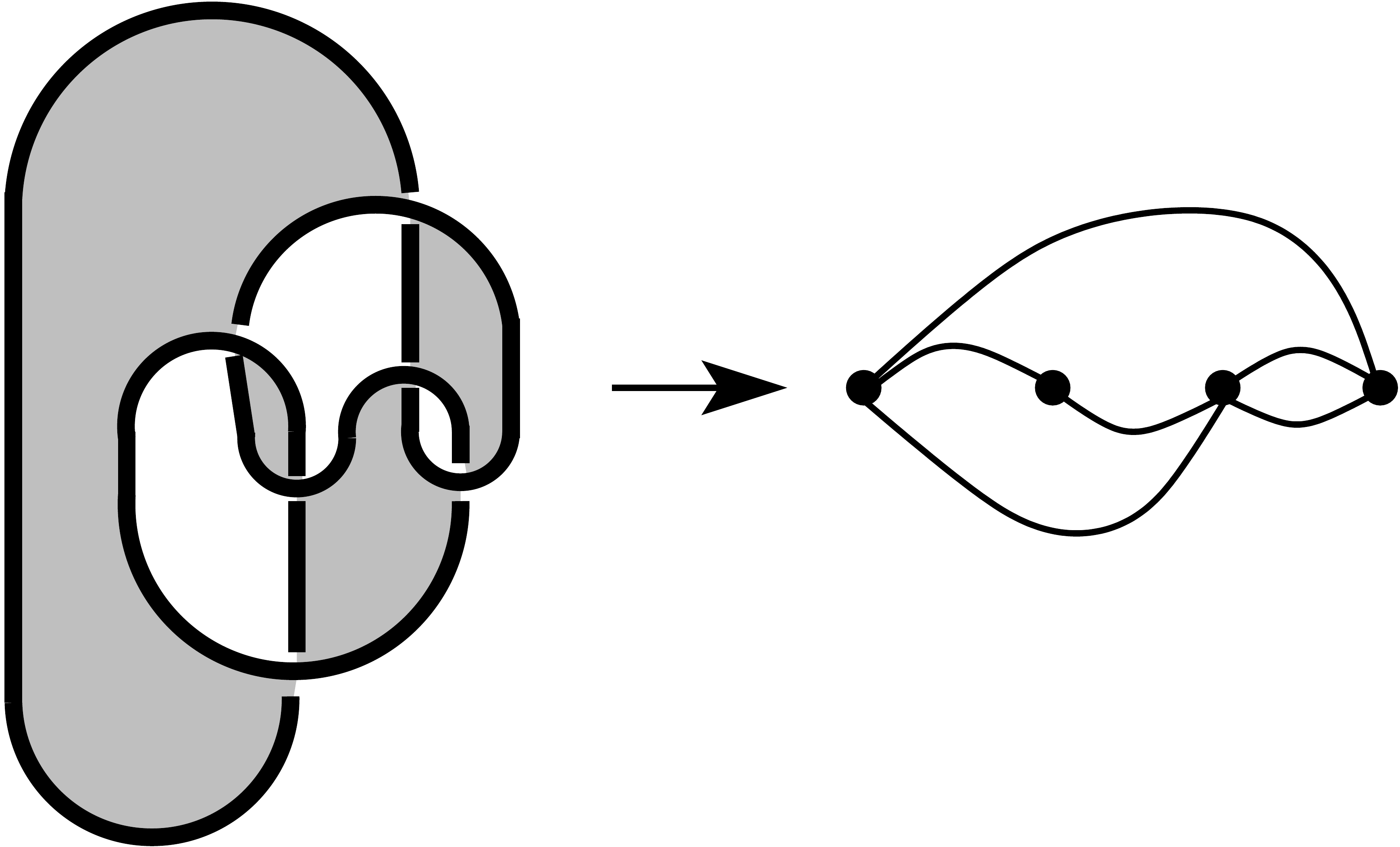} {4in}  $$

Observe that the diagram consists of two (not necessarily ternary) trees, one above and one below the line where the two 
(ternary) trees of $g$ meet.
The strategy of the proof is to take a link diagram $L$ and modify it by planar isotopies and Reidemeister moves 
so that $\Gamma(L)$ looks like a graph as above. 

\begin{definition} A \emph{rooted planar/linear $n$-tree} will be the isotopy class of a planar tree with all 
vertices being points on the $x$-axis, the isotopies being required to preserve the $x$ axis.
The root is then the leftmost point on the straight line which might as well be taken as $0$.
\end{definition}

We see that an element  $g\in F_3$ (defined as a pair of ternary trees) has trees $T^+$ and $T^-$ above and below
the $x$ axis respectively with the $x$ axis as boundary. The signed graphs $\Gamma(T^+)$ and $\Gamma(T^-)$ are both
rooted planar/linear $n$-trees where the trees defining $g$ have $2n-1$ leaves.

\begin{remark} \label{pain} \rm{Here there is a new pheomenon compared to the $F_2$ case of \cite{jo4}.
If $g\in F_2<F_3$ (and there are no cancelling carets) all the edges of $\Gamma(T^+)$ have a plus sign and all the edges of $\Gamma(T^-)$ have a minus sign. This is no longer true for a general element of $F_3$. In fact each edge of $\Gamma(T^+)$ and $\Gamma(T^-)$ can be  orientated pointing away from the root of the tree. For $\Gamma(T^+)$, if the $x$ co-ordinate of the first vertex of the edge is less than the  $x$ co-ordinate of the second then the sign of the edge is plus, and in the opposite case it is minus. And the other way round for $\Gamma(T^-)$.}

(*)Elements of $F_2$ are characterised by the fact that the $x$ coordinate increases along edges.
\end{remark}

\begin{proposition}\label{fc2}There are $FC(3, n-1)=\D\frac{1}{2n-1}\binom{3n-3}{n-1}$ rooted planar/linear $n$-trees with $n$ vertices.
\end{proposition}
\begin{proof} We need to establish the recurrence relation in the proof of \ref{fc1} for $k=3$. Let $pl_n$ be
the number of rooted planar/linear $n$-trees and take a rooted planar/linear $n$-tree $t$ with $2$ vertices. Then  
given $3$ rooted planar/linear $n$-trees $t_1, t_2 $ and $t_3$ one may form another rooted planar/linear $n$-tree
by attaching $t_1$ to the right of the root of $t$, the reflection of $t_2$ in the $y$ axis to the left of the non-root vertex
of $t$ and $t_3$ to the right of the non-root vertex of $t$. Moreover any rooted planar/linear $n+1$-tree can be decomposed
in this way. Thus $pl_{n+1}=\D \sum_{\ell_1+\ell_2+\ell_3=n} pl_{\ell_1}pl_{\ell_2}pl_{\ell_3}$.
 
\end{proof}

By \ref{fc1}, \ref{fc2} and the injectivity of the map from trees to tangles, or directly, there are the same number of 
rooted planar/linear $n+1$-trees as there are ternary $n$-trees, and given two rooted planar/linear $n+1$-trees 
$\Gamma_\pm$ we can construct an element of $F_3$ by flipping $\Gamma_-$ upside down and attaching it
underneath $\Gamma_1$ to form a planar graph and signing all the edges according to whether their end points
have smaller or larger $x$ coordinate, we obtain a signed planar graph $\Gamma_+\cup \Gamma_-$ from which we get a link $L$ with
$\Gamma(L)=\Gamma_+\cup \Gamma_-$. 

 By remark \ref{pain}, if $g\in F_2$, any tree in $\Gamma_+\cup \Gamma_-$ that arose at this point were rooted planar/linear $n$-trees of a special kind-namely any vertex was 
connected by exactly one edge to a vertex to the left of it. Such trees are  counted by the usual Catalan numbers.

To prove theorem \ref{alexander}, we see that it suffices to find rooted planar/linear $n+1$-trees 
$\Gamma_\pm$ so that $\Gamma_+\cup \Gamma_-$ differs from $\Gamma(L)$ by planar isotopies and
Reidemeister moves.  We will give a slightly improved version of the argument of \cite{jo4} which will give us
elements of $F_2$. Surprisingly, we will only use Reidemeister moves
of types I and II.

Note that we may suppose the link diagram $L$ is connected so that $\Gamma(L)$ is too.

First isotope $\Gamma(L)$ so that all its vertices are on the $x$ axis. 
Unless there is a Hamiltonian path through the vertices of $\Gamma(L)$,  there will be edges of $\Gamma(L)$ the $x$ axis. Taking care of these edges  is very simple and is described in lemma 5.3.6 of \cite{jo4}, but some previous 
versions of \cite{jo4} are missing this point-near where the offending edge cross the $x$ axis, just add two vertices to
$\Gamma(L)$ on the $x$ axis and join them with an edge signed $\pm$  according to the sign of 
the offending edge. Continue to call this graph $\Gamma(L)$. See \cite{jo4}. 

At this stage we have a lot of the ingredients of an edge-signed rooted planar/linear graph of the form $\Gamma(L(g))$. 
$\Gamma(L)$ consists of two graphs, $\Gamma(L)^+$ and $\Gamma(L)^-$ in the upper and lower half-planes
respectively, with vertices all lying on the $x$ axis. The root is the vertex with smallest $x$ coordinate. We will make a
series of modifications and continue to call the graph $\Gamma(L)$ after each modification since it will represent the
same link.

The first thing we will take care of is the signs. Since we are trying to produce an element of $F_2$, we must end
up with all $\Gamma(L)^+$ signs positive and all $\Gamma(L)^-$ signs negative. There is no reason for this to be 
true.  But we may change $\Gamma(L)$ by type II Reidemeister moves so as to correct the bad signs one at a time.
Here is how-in the diagram below the dashed lines are edges of indetermate sign, the solid lines with no signs 
are positive edges if they are above the $x$ axis and negative if below, except for a solid line with a sign next to it
which is an edge with that sign.

$$ \vpic {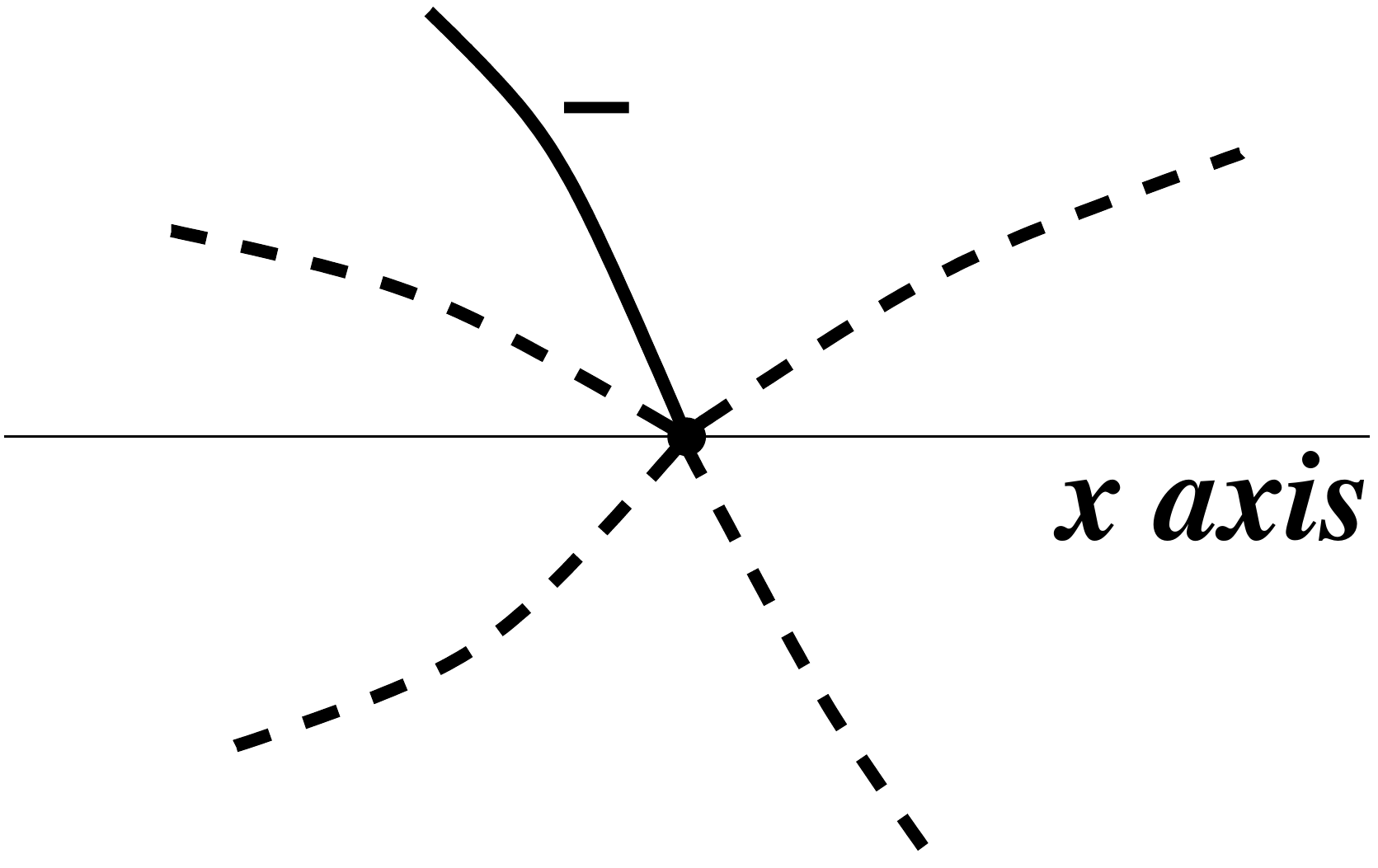} {2in}  \qquad \rightarrow \vpic {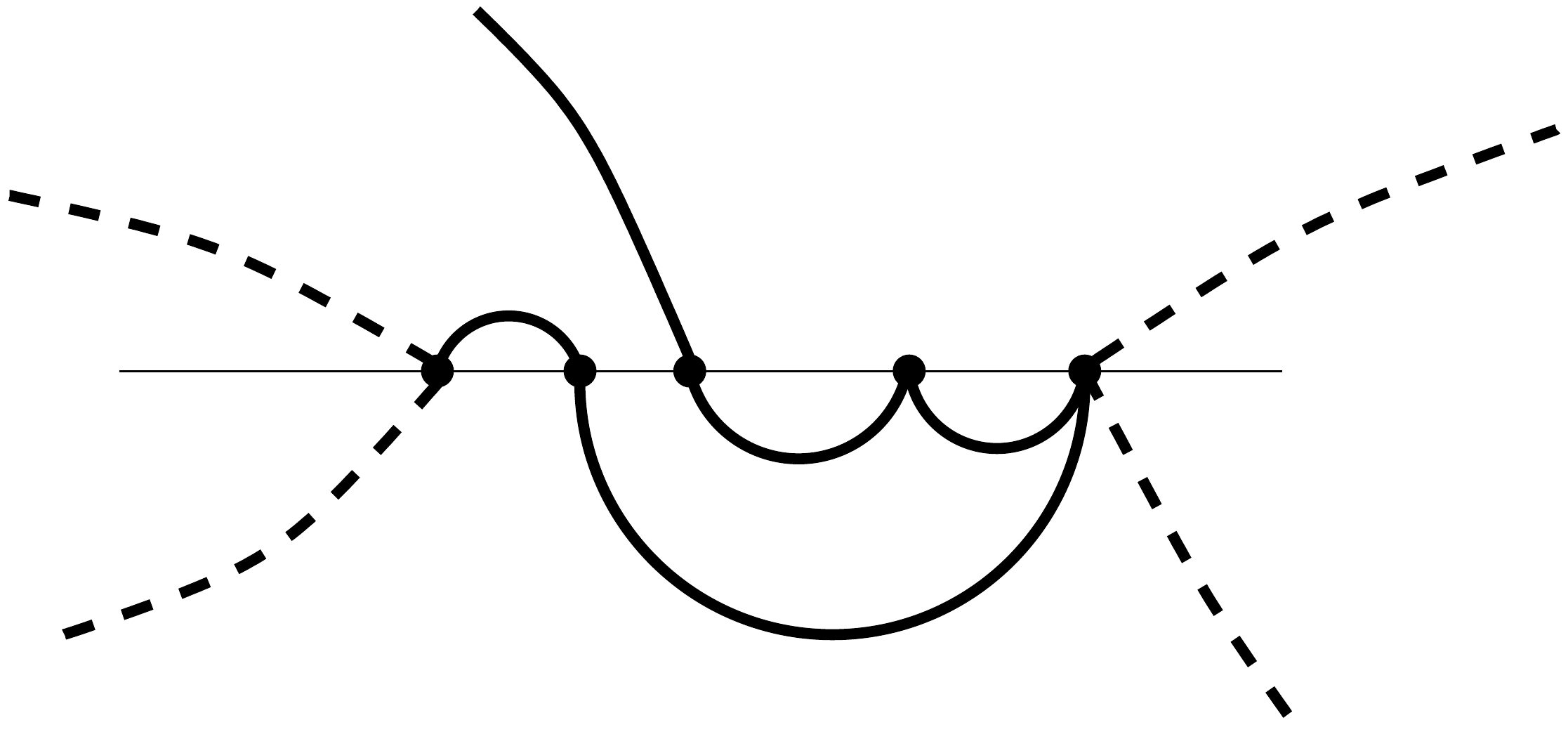} {2.5in} $$

Here we have started with a ``bad'' edge above the $x$ axis and changed the graph near one end of that edge. The two small
added edges in the lower half plane cancel (type II Reidemeister move) with the solid edge above to recreate 
the bad edge. The other two added edges just cancel to return the picture to its orignal form. Thus that part of
the graph $\Gamma(L)$ shown on the left gives the same link after being replaced by the figure on the right.

We now need to alter the graph so that if we orient the edges away from the root then the $x$ coordinate of their
sources are less than that of their targets. We can say this informally as : ``each vertex is hit exactly once from 
the left'', top and bottom. If we do this with local changes in accordance with out convention that edges in the
upper and lower half planes are positive and negative respectively, we will be done.

First let us make sure that every (non root) vertex is hit from the left, top and bottom. If there is one that is not, simply join
it to its neighbour on the left with a pair of cancelling edges thus:
$$ \vpic {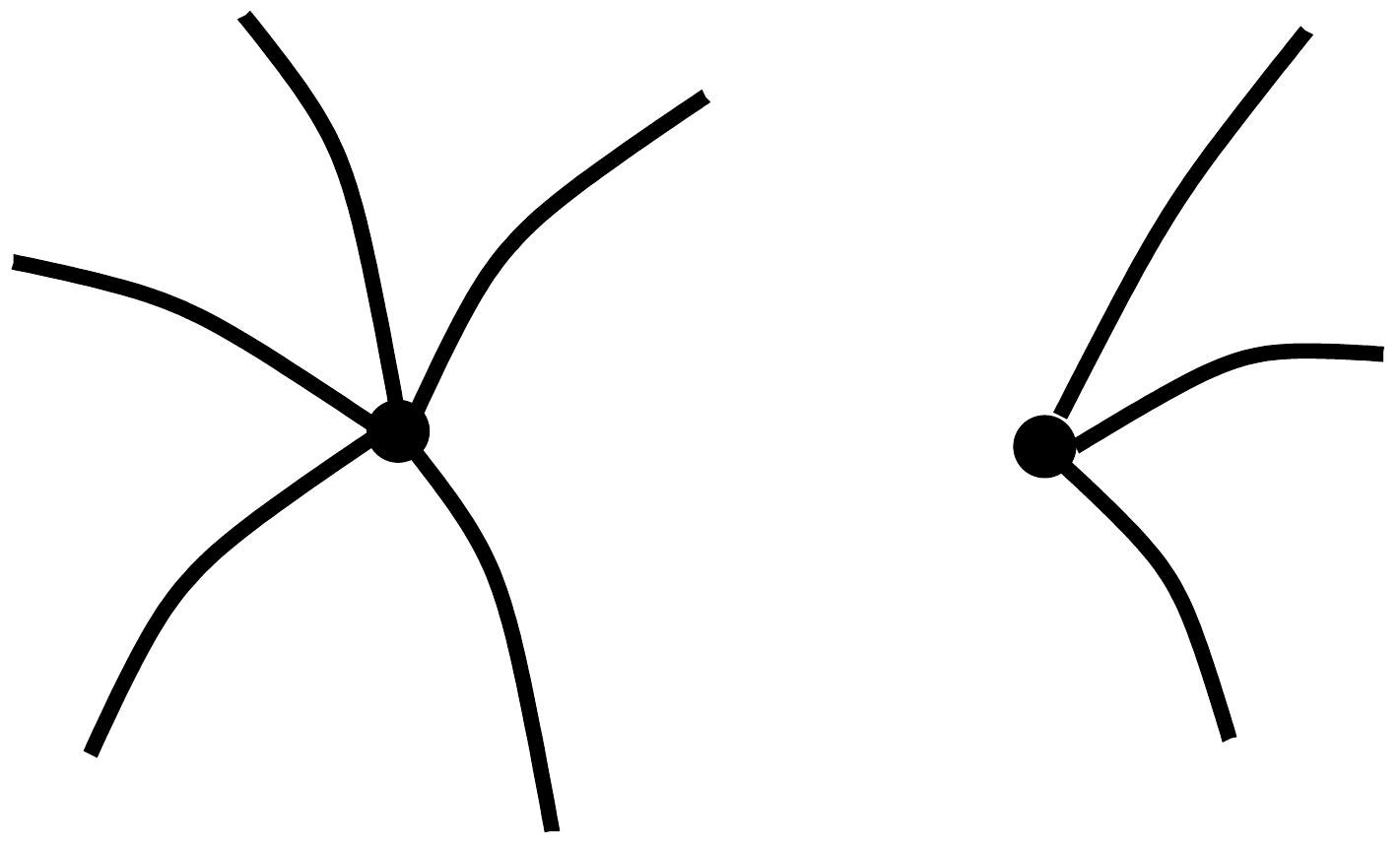} {1.8in} \qquad \rightarrow \qquad \vpic {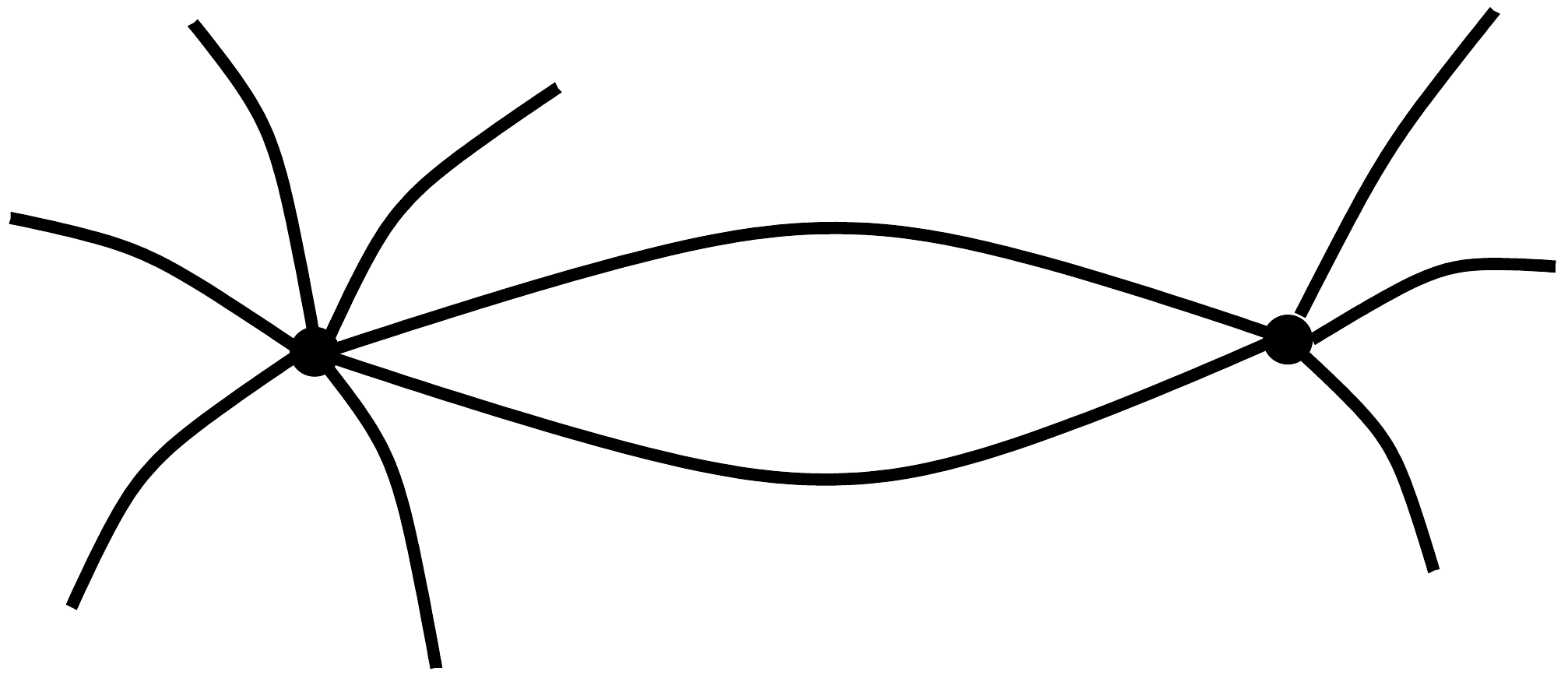} {2.5in}  $$
(Recall that the edges all get their signs from being in the upper or lower half plane.)

Now our only remaining problem is that vertices may have multiple hits from the left. We only need to show how
to get rid of the leftmost one at each vertex, wolog in the upper half plane.  Proceed thus:

$$ \vpic {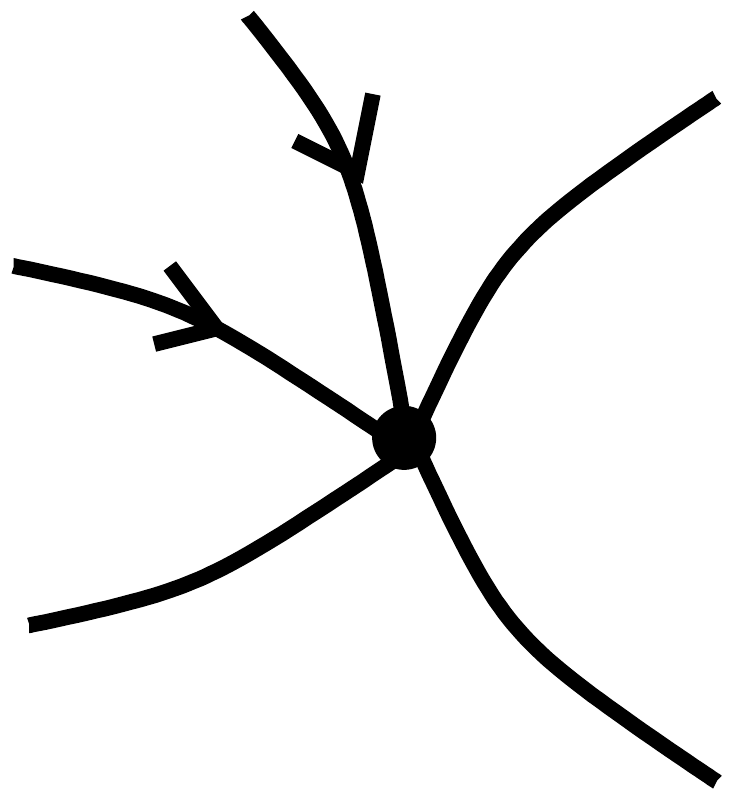} {1in} \qquad \rightarrow \qquad \vpic {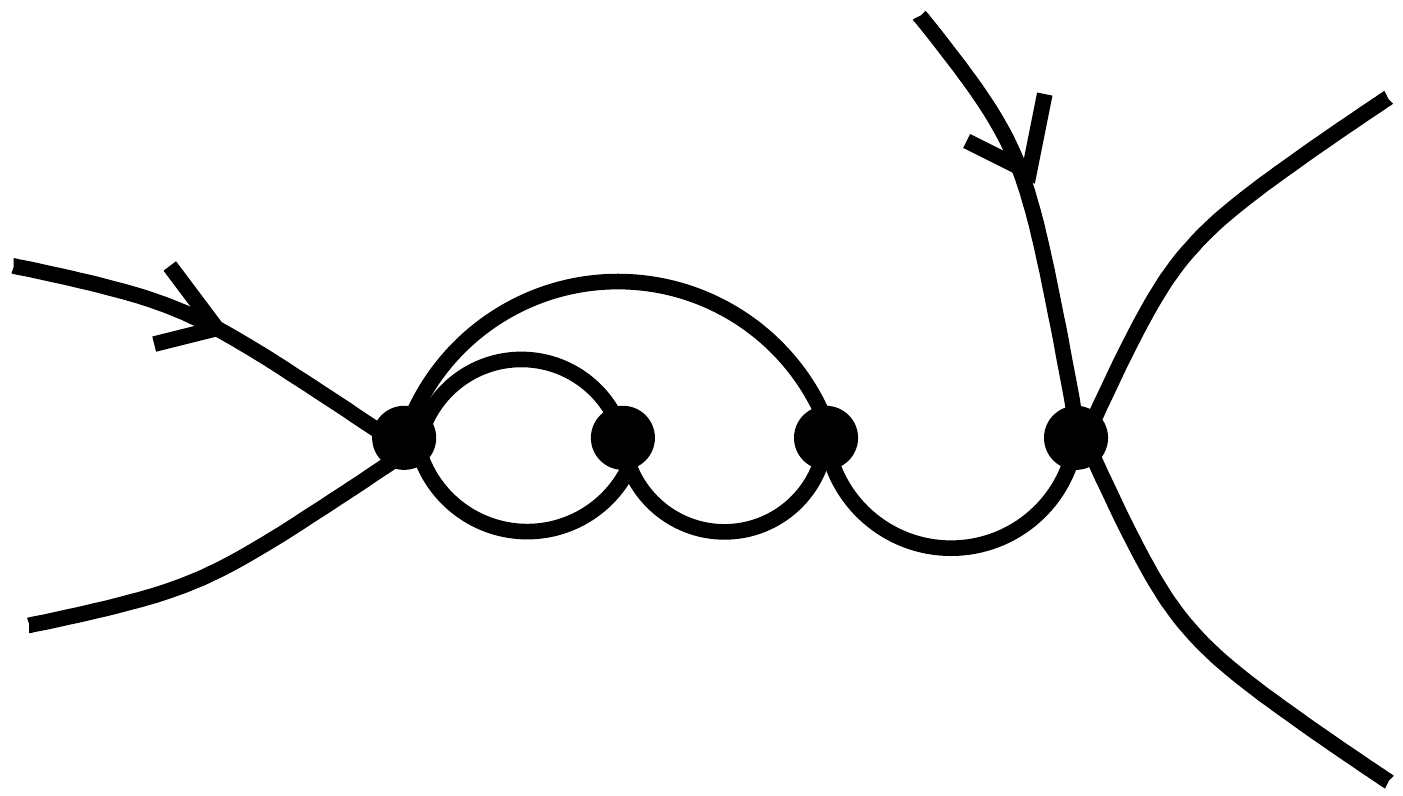} {1.5in}  $$

The two edges hitting the vertex from the left have been replaced by one. A brace of type I and II 
Reidemeister moves shows that replacing the picture on the left by the one on the right preserves the
link $L$. And all the added vertices are hit exactly once, top and bottom, from the left.

Continuing in this way one obtains a $\Gamma(L)$ consisting of  two rooted planar/linear $n$-trees top
and bottom from which an element of $F_3$ (in fact it's in $F_2$) may be reconstructed by the method we
described for going from a signed planar graph back to a link diagram.

This ends the proof of theorem \ref{alexander}.
\end{proof}

The algorithm for constructing Thompson group elements from links in the above proof is of theoretical interest
only. In particular the sign correcting move is very inefficient. Even for the Hopf link, if one starts with the following
$\Gamma(L)$ and applies the algorithm, the Thompson group element is very complicated (remember that the top
edge is a positive one by convention):
$$\vpic {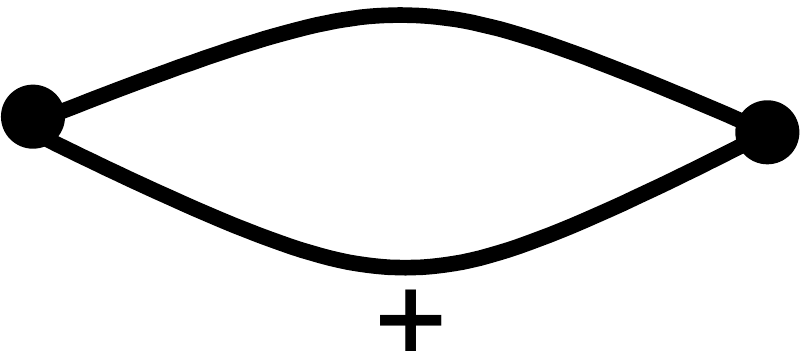} {1.2in} $$

\section{The annular version, Thompson's groups $T_n$.} 

There are two other well known versions of the Thompson groups $F_n$, namely $T_n$ and $V_n$. 
$T_n$ is a group of PL homeomorphisms of the circle (rather than the interval) with slopes all powers of $n$
and non smooth ponts all of the form $\D \frac{a}{n^b}$. $T_n$ contains an obvious copy of $F_n$ and can be
obtained from it by adding rotations of the circle by angles $\D \frac{2\pi a}{n^b}$.
$V_n$ is even bigger, allowing discontinuous permutations of the intervals on the circle. 

Both $T_n$ and $V_n$ can be constructed from our category of forests method by suitably decorating the
forests with cyclic and general permutations respectively-see \cite{nogo}. But the functor to tangles only works for
$T_n$ because of the discontinuities in $V_n$. Obviously all knots and links can be obtained from $T_3$ from this functor
since they can already be made from $F_3$, but some links may be much easier to realise using $T_3$.

There is a bigger group called the ``braided Thompson group'' \cite{braidedthompson} which should have all the advantages of both
braids and the Thompson groups.

\section{The group structure-analogy with braid groups.}

We are promoting the Thompson groups as groups from which links can be constructed, like the braid groups.
In this section we will establish a strong, albeit not always straightforward, analogy between the two groups and their
relationships with links. 

The most obvious first thing missing from our construction of the Thompson groups in section \ref{construction}, which is front and centre in the braid 
groups, is a gemoetric understanding of the group law. But this is supplied by work of Guba and Sapir in \cite{gus} and Belk in \cite{belk}.
Here is how to compose two $F_3$ elements, given as pairs of rooted  planar ternary ternary trees $\frac{s}{t}$ as
usual, from this point of view.

Given $\D \frac{r}{s}$ and $\D \frac{s}{t}$, draw them as we have with the denominator on the bottom and the numerator, upside
down, on the top, joined at the leaves, thus:

$$ \vpic {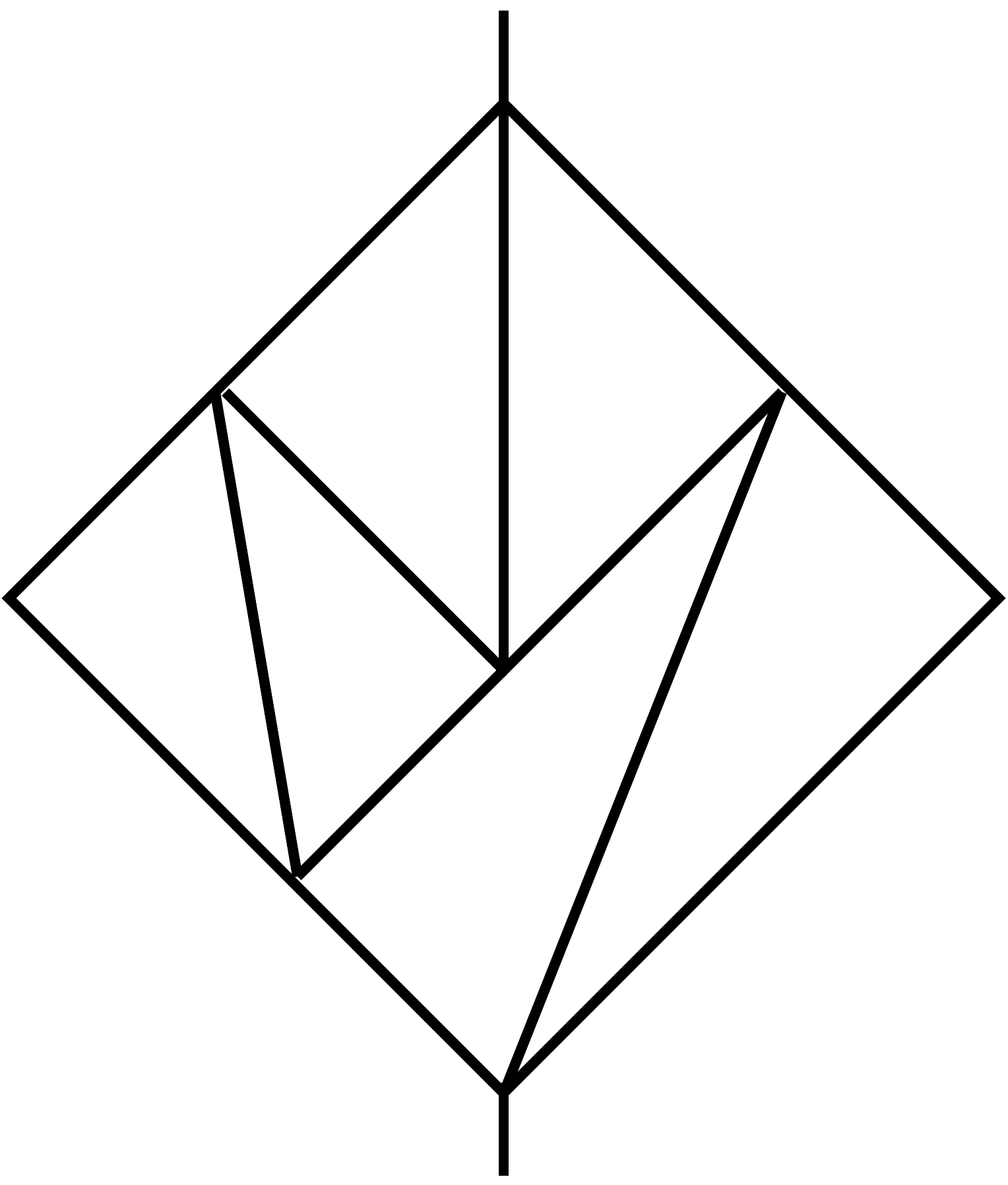} {1.5in} $$

Now arrange the picture of $\D \frac{s}{t}$ underneath that of $\D \frac{r}{s}$ with the top vertices aligned, and 
fuse the top edges thus:

$$ \vpic {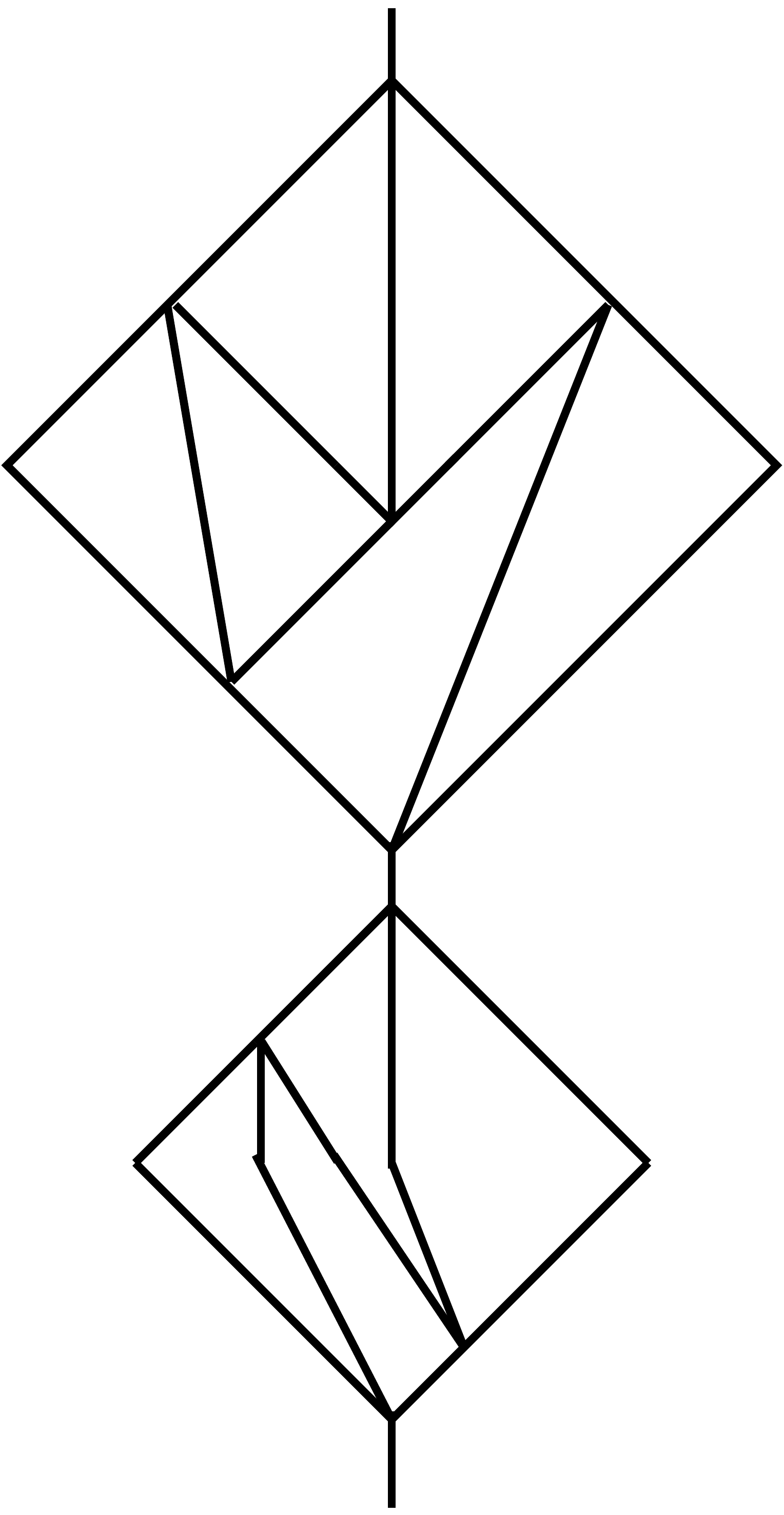} {1in} $$

Now apply the following two cancellation moves:

$$(i) \vpic {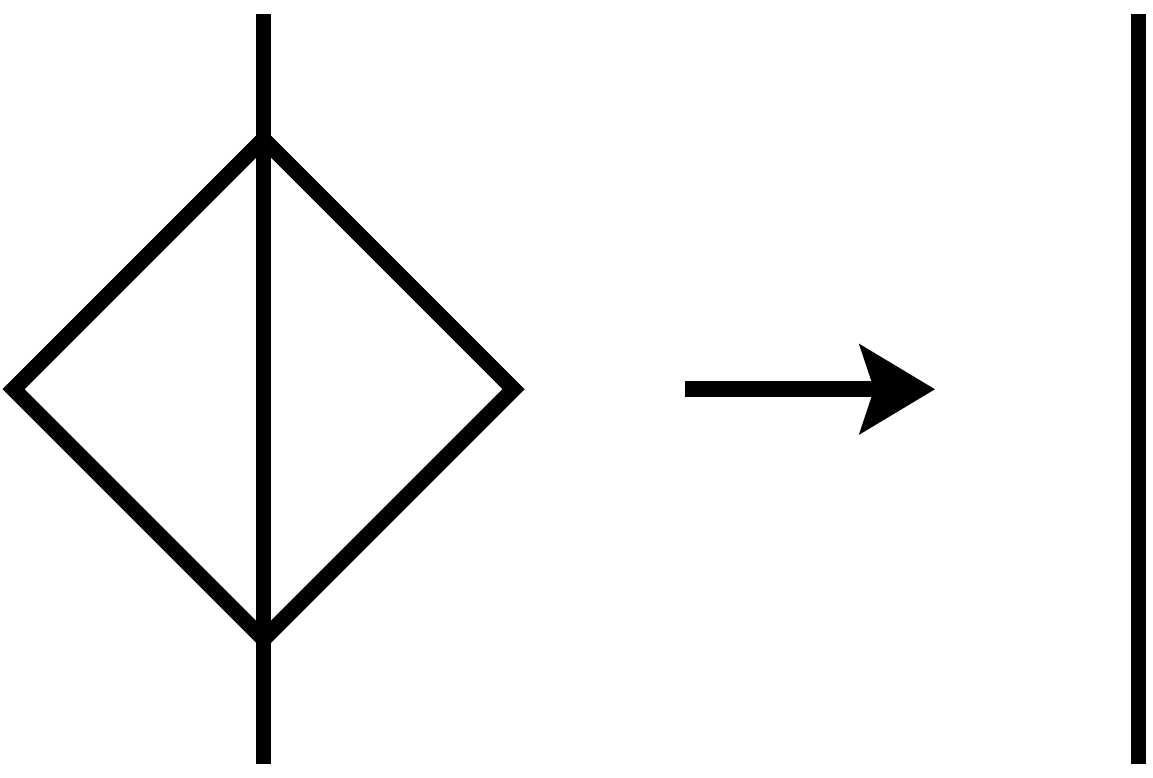} {1in} \mbox{     and    }  (ii)  \vpic {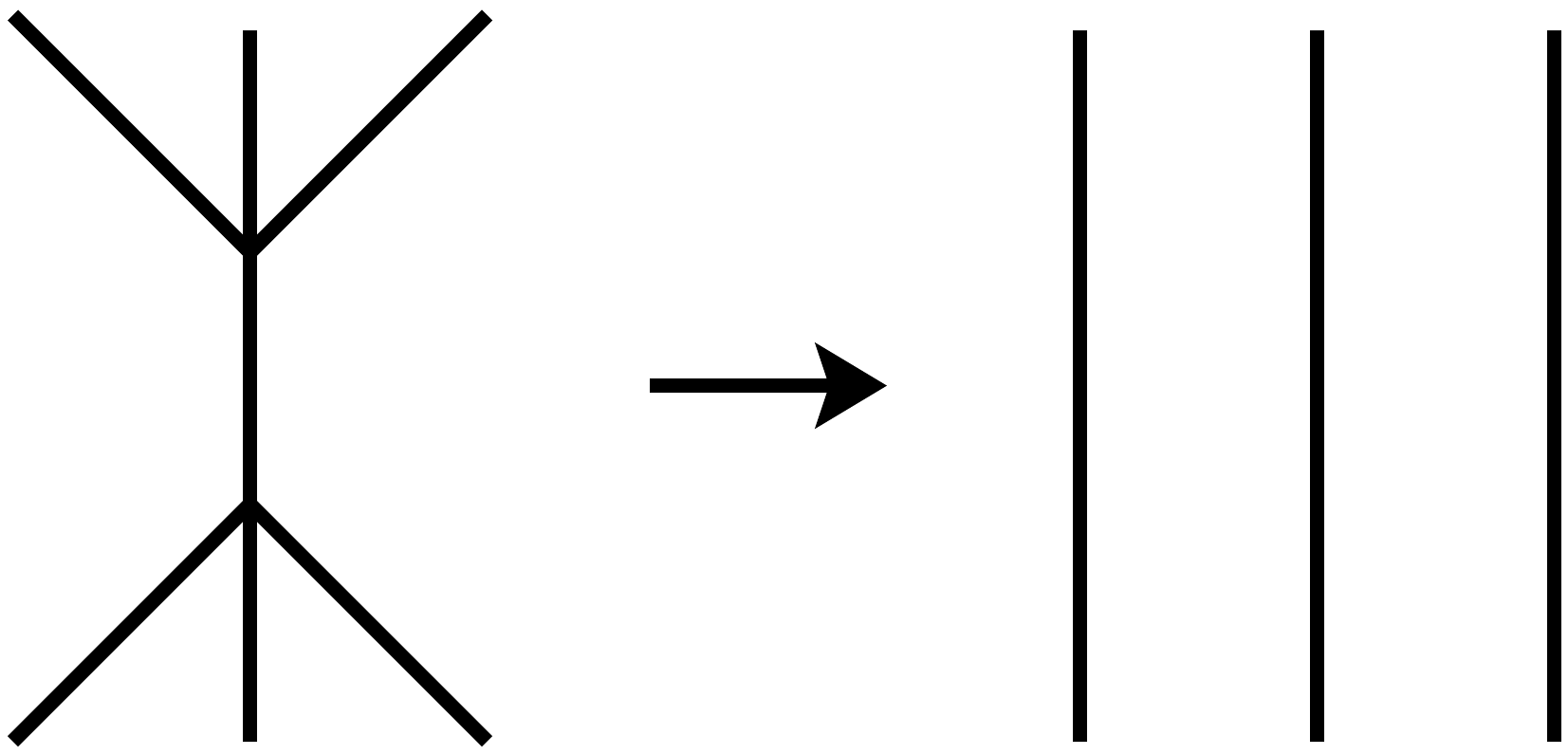} {1.5in} $$ 
until they can no longer be applied. It is easy enough to see that at this point the remaining diagram
can be decomposed into a pair of ternary planar trees, thus another element of $F_3$. We illustrate with the above 
example:
$$ \vpic {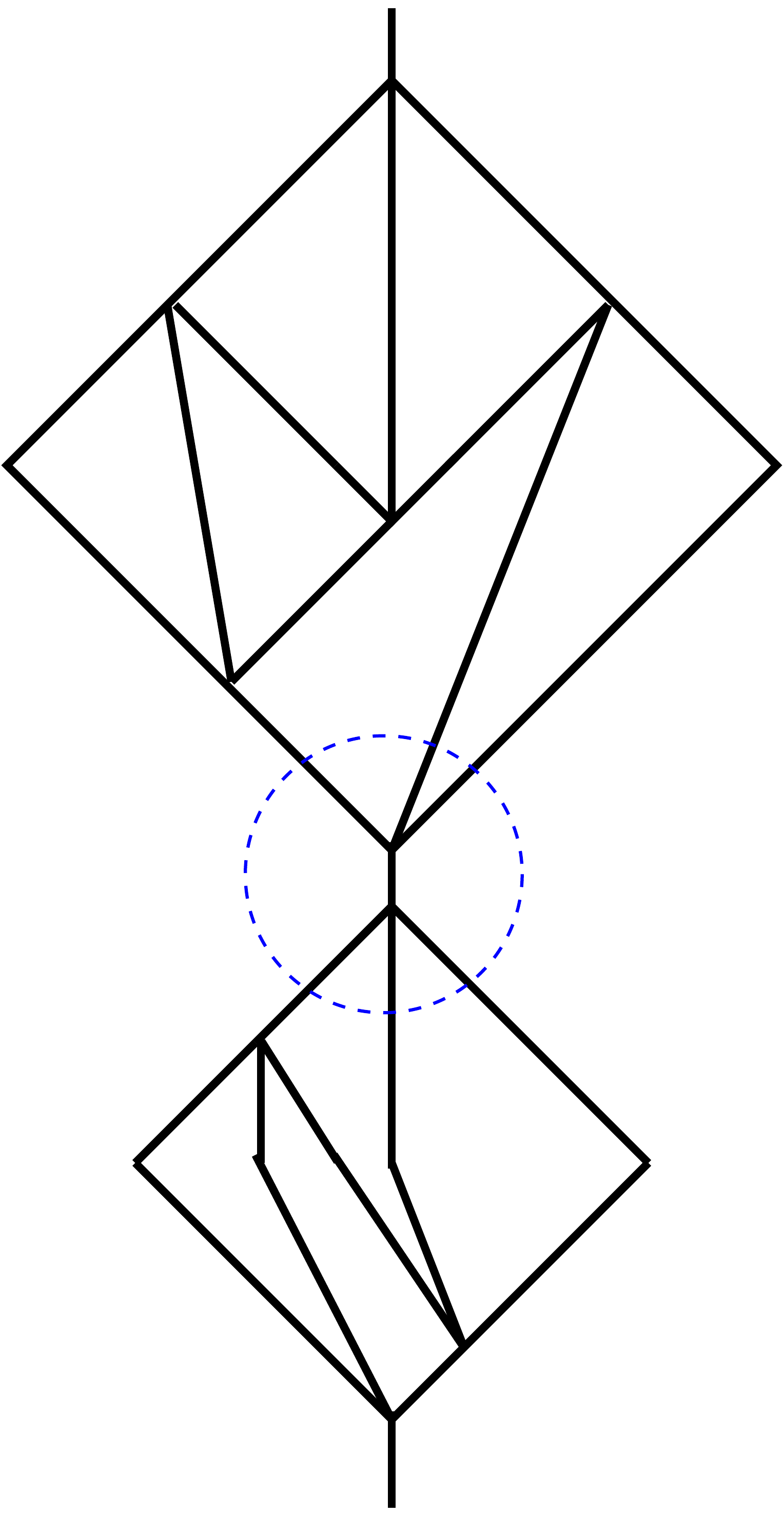} {0.8in}  \quad \rightarrow \quad   \vpic {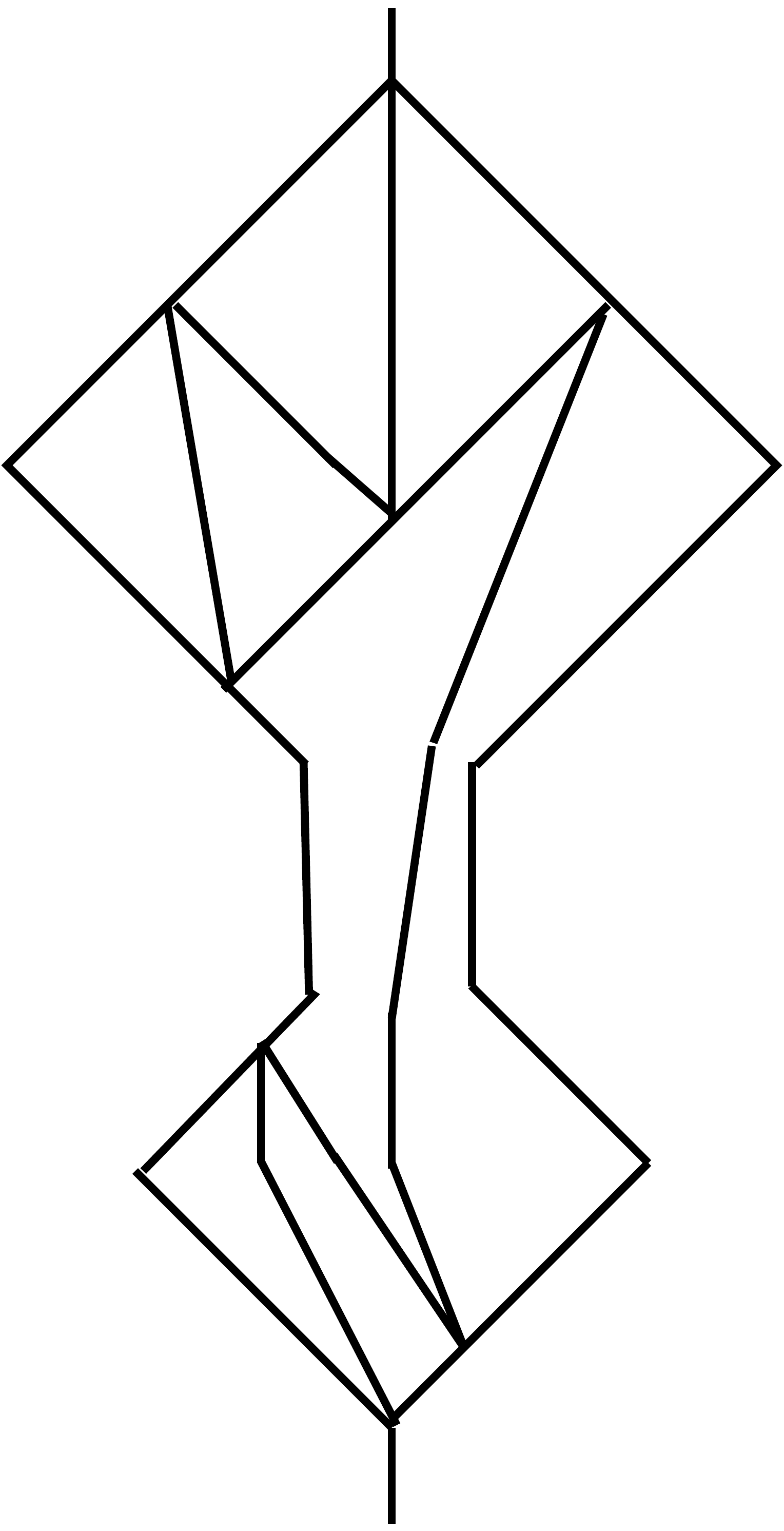} {0.8in}  \quad \rightarrow \quad   \vpic {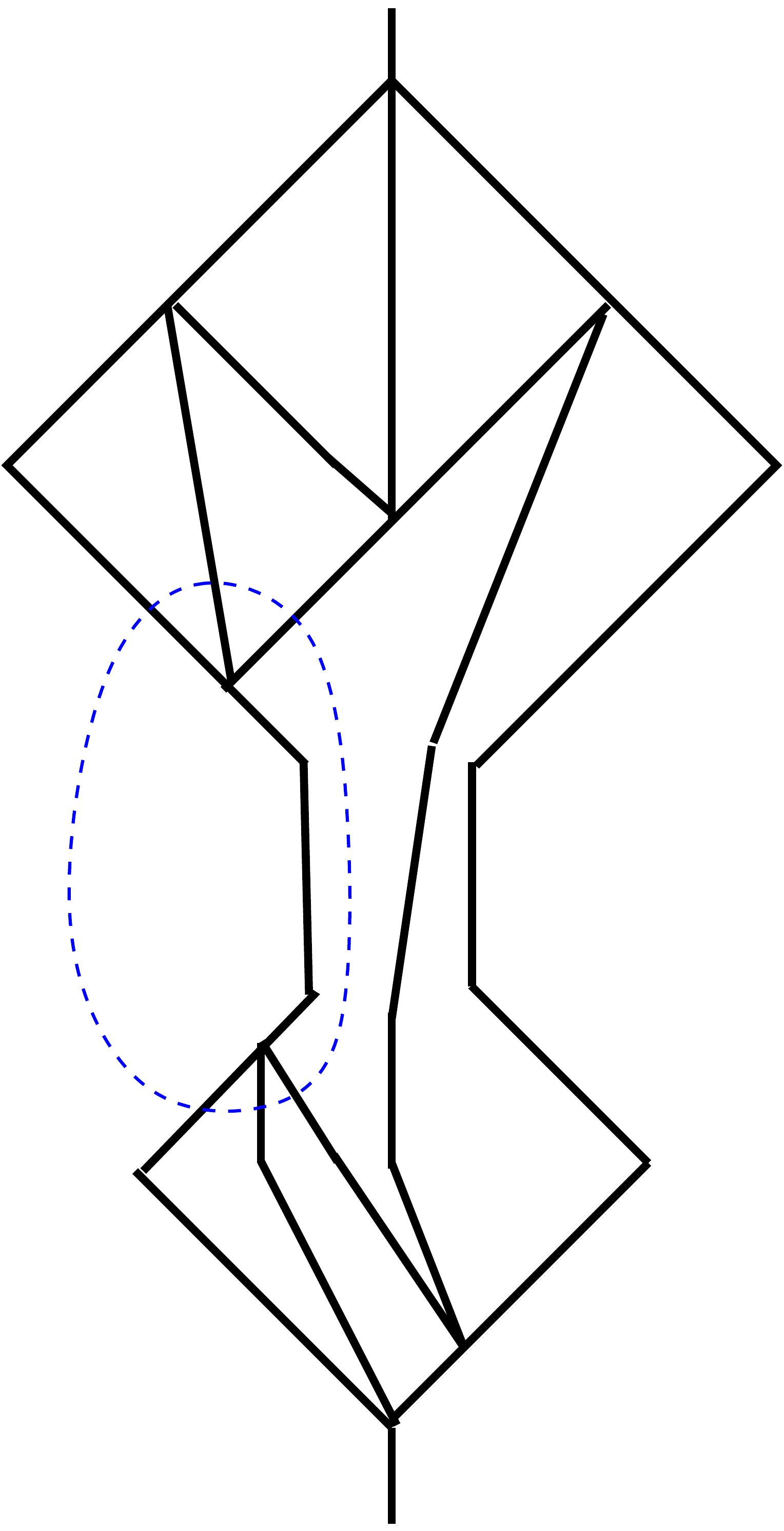} {0.8in} \quad \rightarrow \quad   \vpic {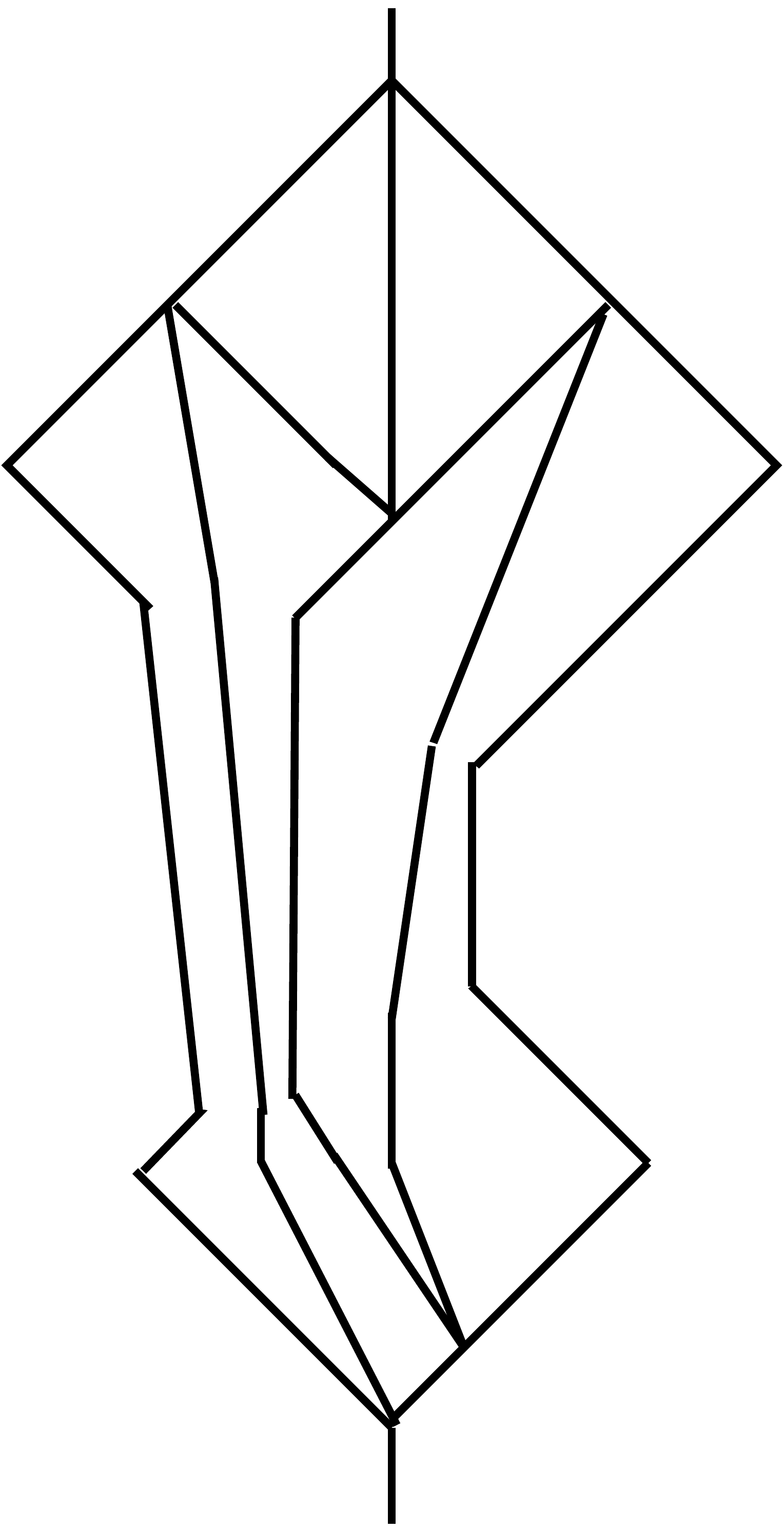} {0.8in} $$

Now we draw in a curve showing the split between the top and bottom trees, and redraw as a standard picture.

$$ \vpic {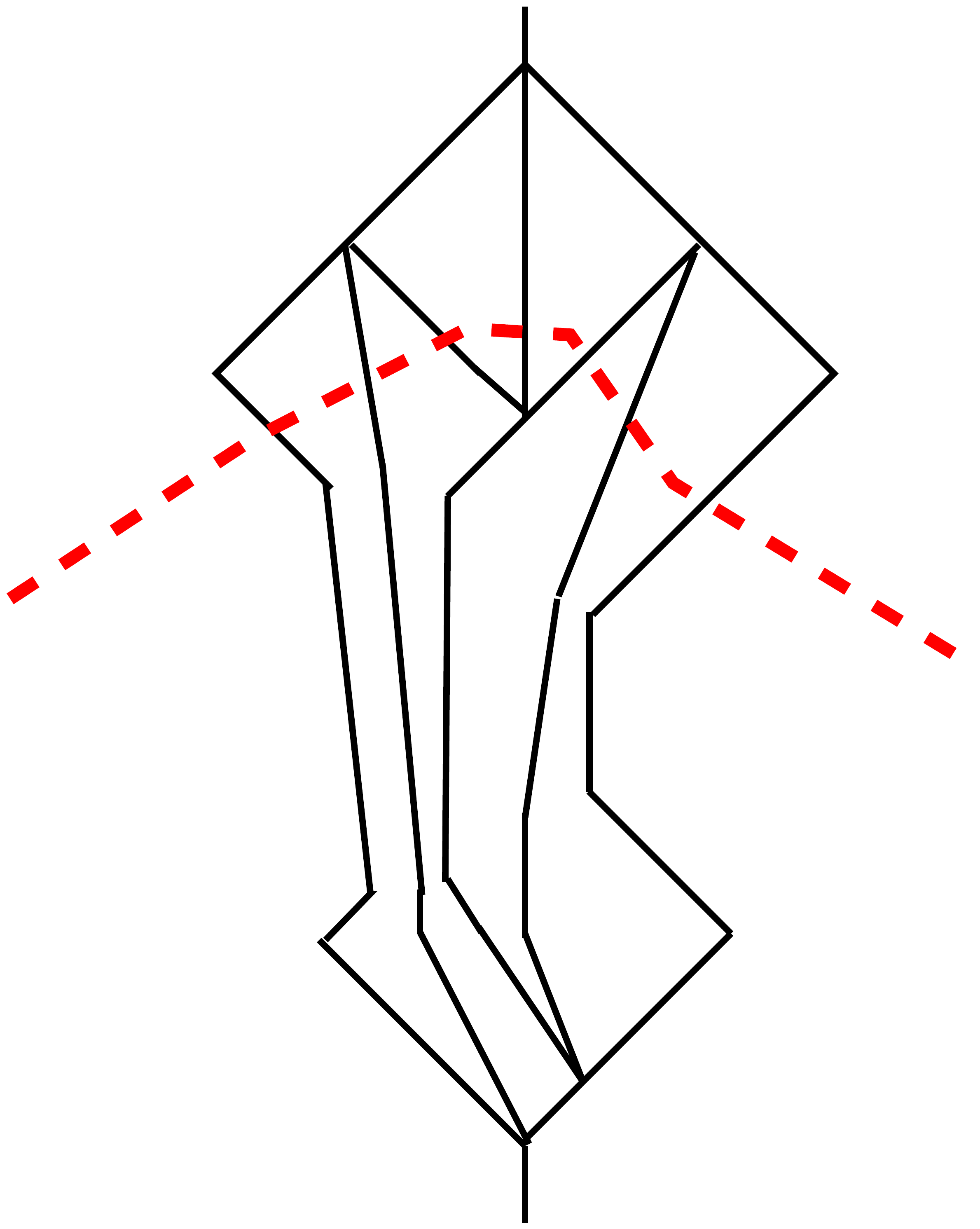} {1.4in}   \rightarrow  \vpic {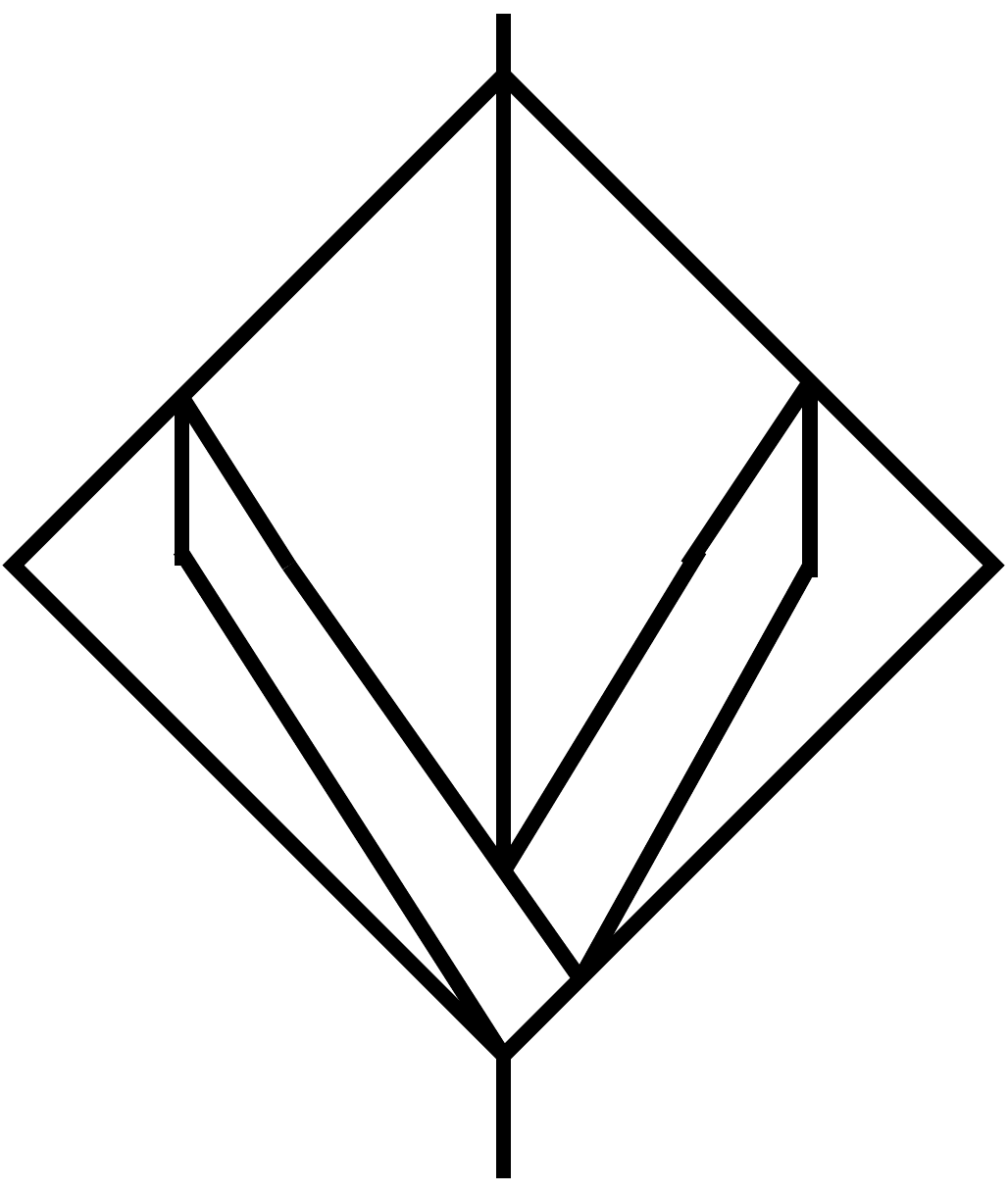} {1in}  =  \vpic {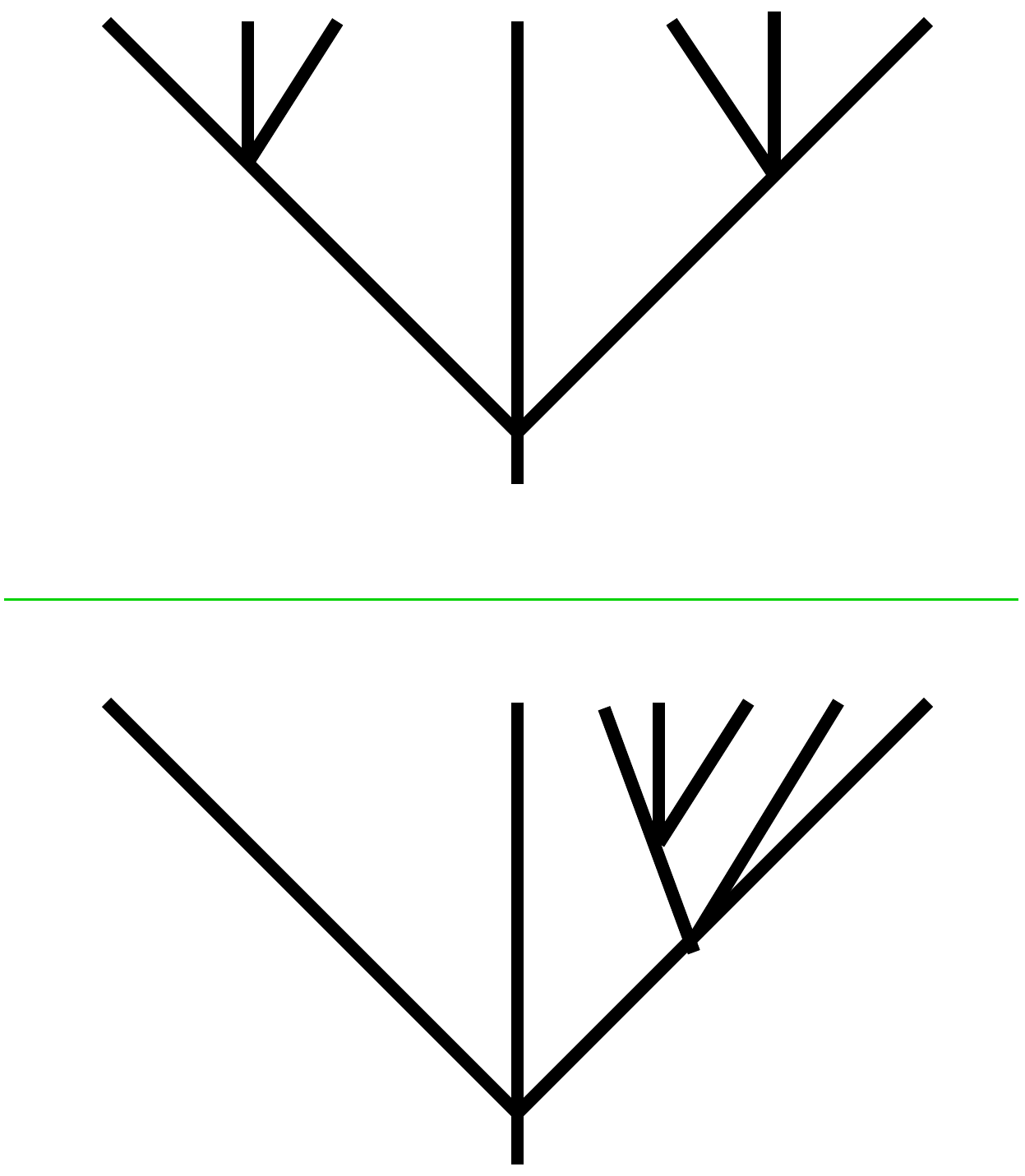} {1.3in} $$

N.B. \emph{ It is important to note that of the vacuum expectation value of the composition- by replacing the
vertices of the diagrams by crossings- cannot be done until all the cancellations have been made.}

This is because the knot theoretic move \vpic{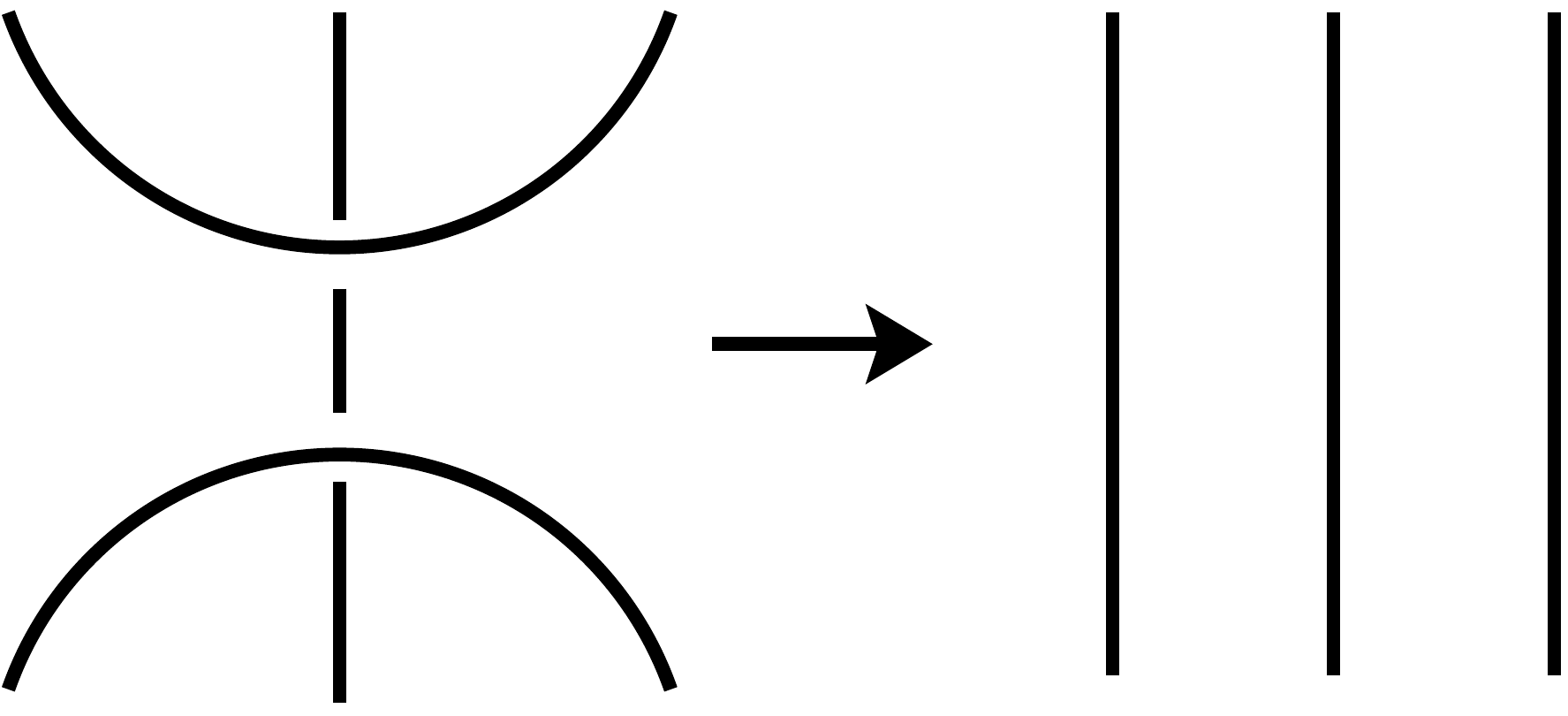} {1in} is NOT an isotopy. It is what we call and ``elementary cobordism''. 

\begin{definition} Changing \vpic {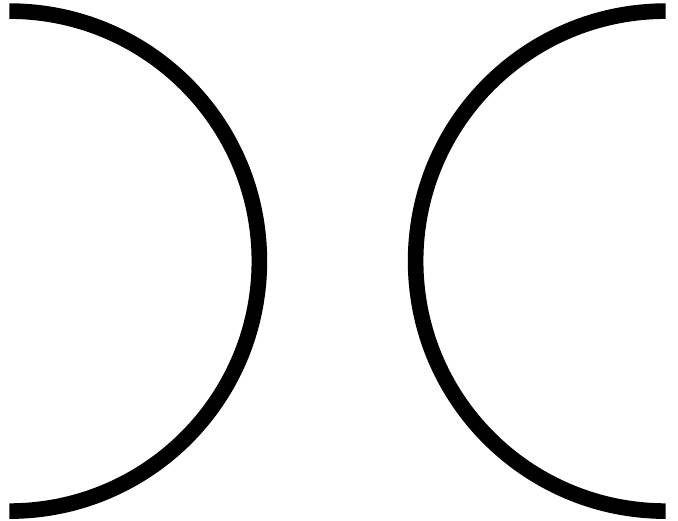} {0.3in} to \hpic {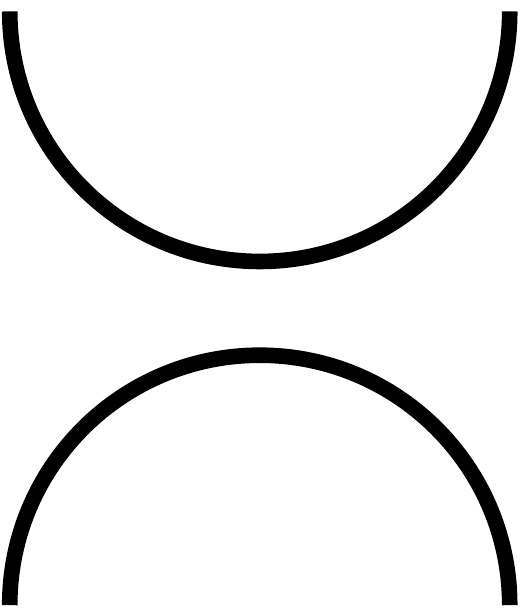} {0.3in} in 
a link diagram is called an  ``elementary cobordism''.
\end{definition}

Thus each time we apply the cancellation move (ii) above we are changing the undelying link by an elementary 
cobordism. In particular we see that the first cancellation applied in the sequence of moves in the Guba-Sapir-Belk
composition method actually transforms the underlying link almost into the connect sum-it differs from it by
a single elementary cobordism.

Let us call $H$ the subgroup of $F_3$ consisiting of all elements of the form $$\vpic {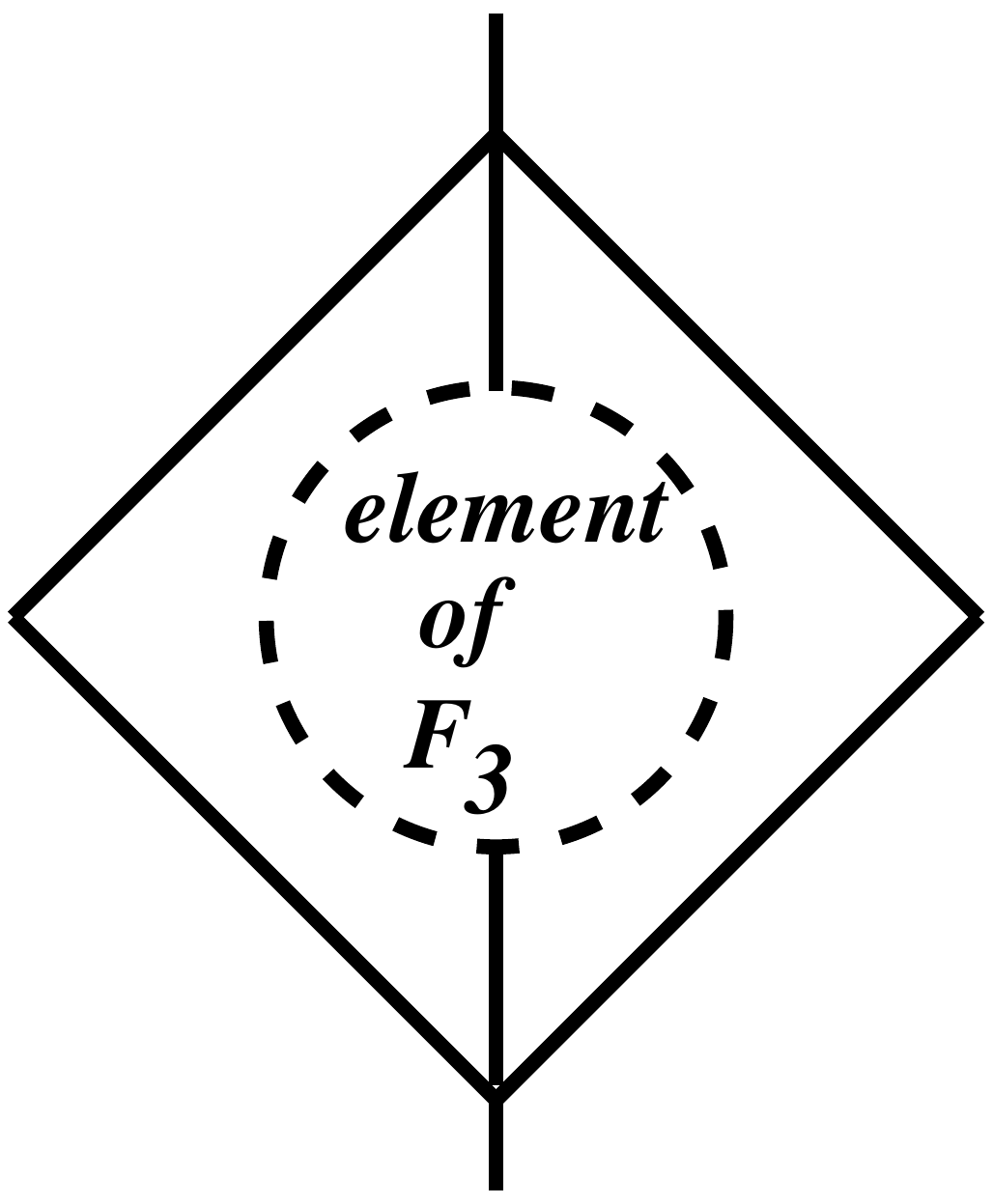} {1in}  . $$ (Obviously $H\cong F_3$.)
For $h\in H$, $L(h)$ always contains a distant unknot sitting on top of another link. Let $\tilde h$ be
$L(h)$ with this distant unknot removed.
Let us also define the size $|g|$ of an element $g\in F_3$ to be the number of vertices in a tree of a minimal
pair of trees picture of $g$, not counting the root vertex. Then $|gh|\leq |g|+|h|$ with equality only if there is 
only the first cancellation when using
the Guba-Sapir-Belk composition.
\begin{proposition} \label{gh} Let $g$ and $h$ be elements of $H<F_3$ with $|gh|= |g|+|h|$. Then 
$$L(\widetilde {gh})=L(\tilde g)\#L(\tilde h)$$
\end{proposition}
\begin{proof}
This is immediate on drawing a picture of $gh$.
\end{proof}

Thus group composition in $F_3$ can be \emph{directly} related to the connected sum of the links. 

\5
We now remind the reader of the two ways links can be obtained from the braid groups $B_n$. They are the trace closure:
$\D \vpic{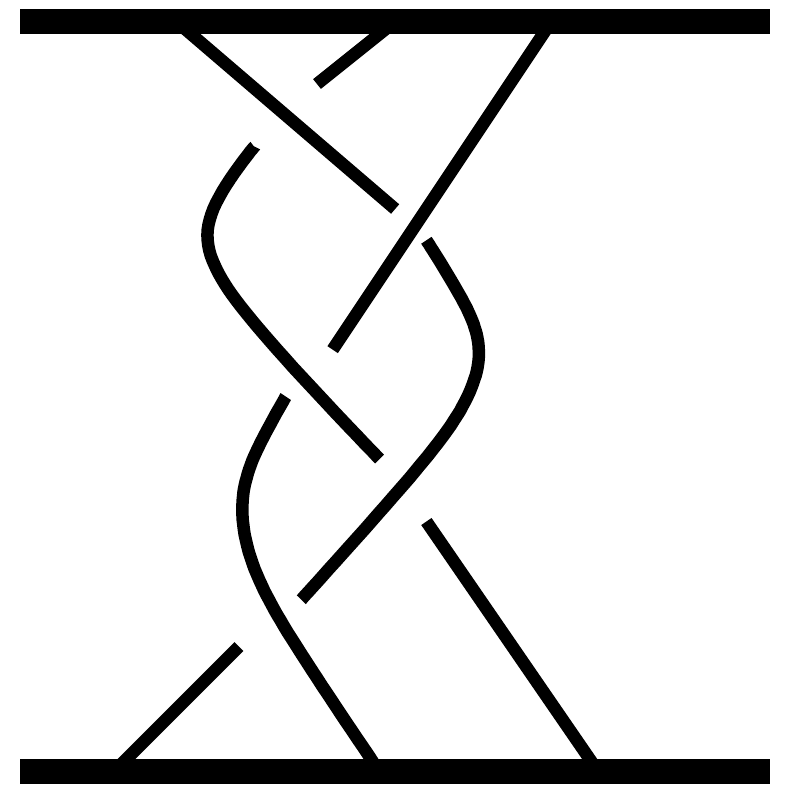} {0.5in} \rightarrow \vpic {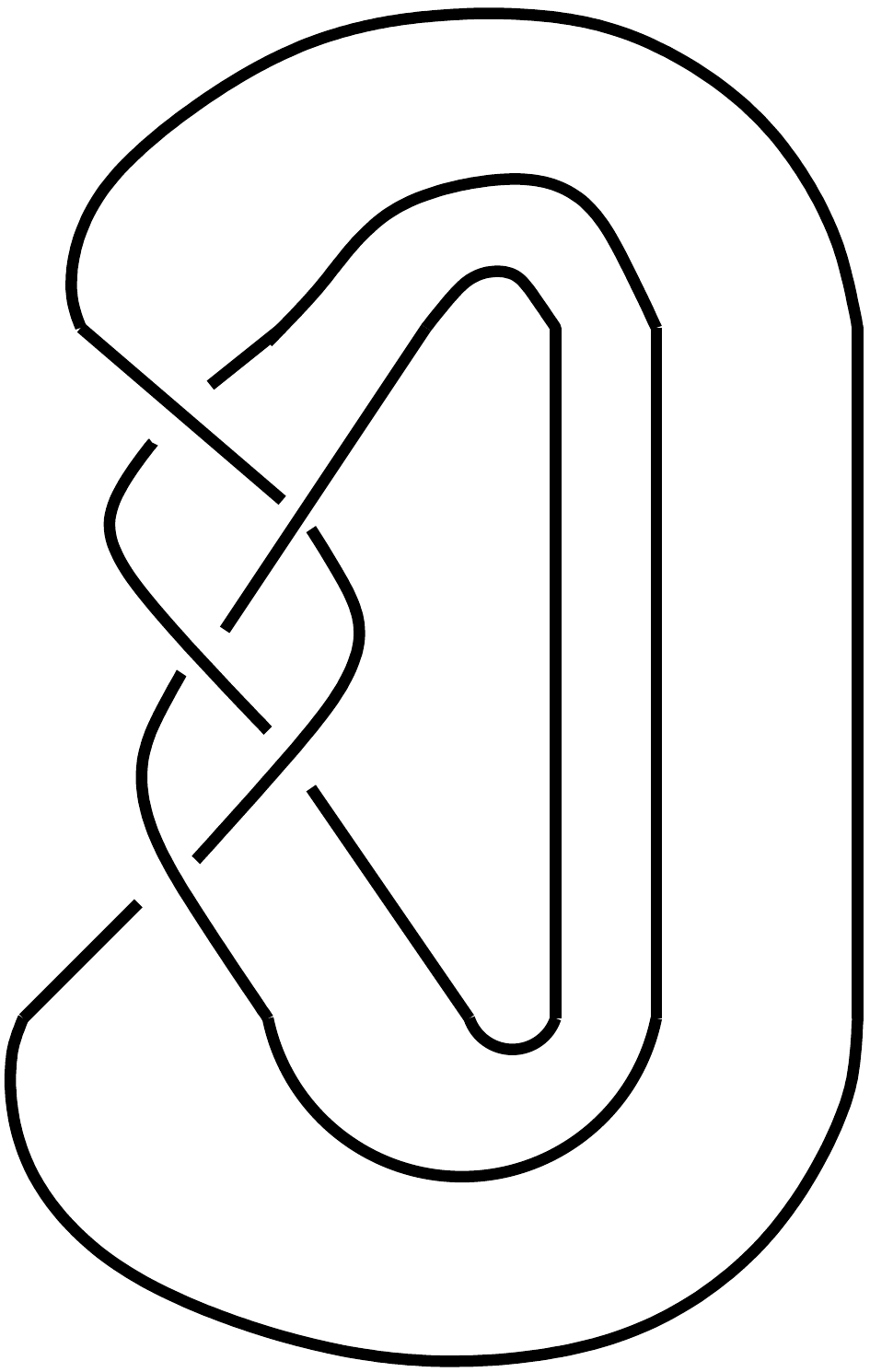} {0.5in} $ and the plat closure $\D \vpic{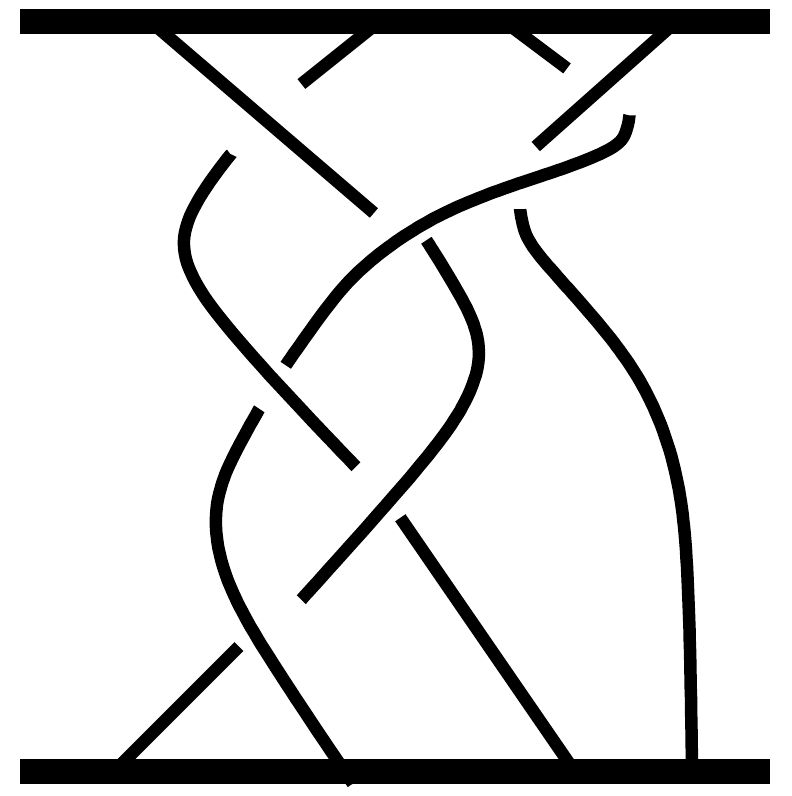} {0.5in} \rightarrow \vpic {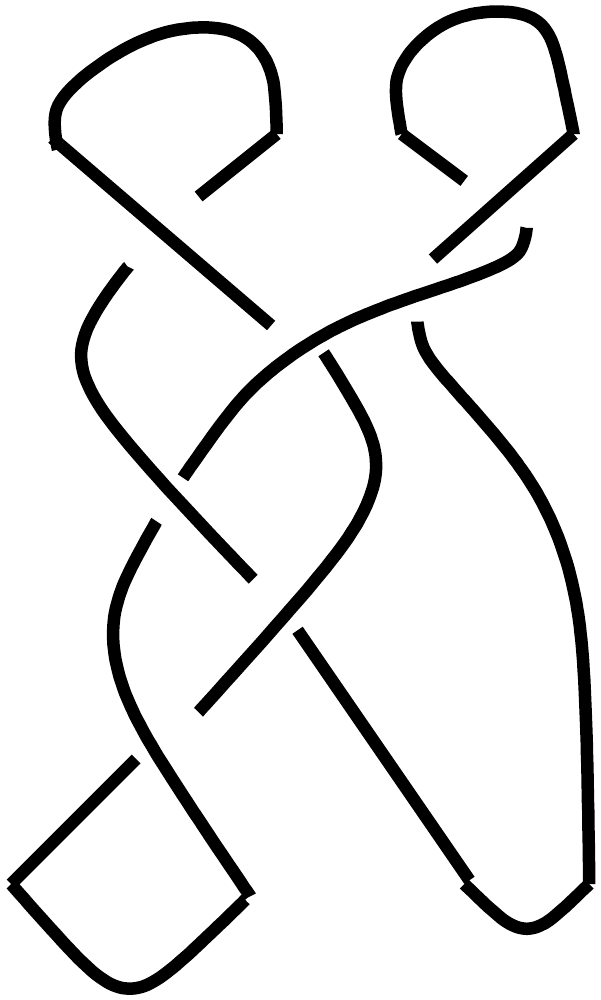} {0.5in} $. The first produces oriented links and the second is inherently unoriented. (We call the 
first the ``Trace'' closure rather than just the closure to distinguish it from the plat closure.
Both closures will produce an arbitrary link for a sufficiently large $n$. (The closure result is a theorem of Alexander \cite{alexander}, the plat
closure exisitence follows fairly obviously from an $n$-bridge picture of a link.) In both cases it is known exactly when two different braids give the same link-for the closure this
is a theorem of Markov (see \cite{birman1}) and for the plat closure a theorem of Birman \cite {birman2}.

We can now create a table comparing and contrasting the braid groups and Thompson group as link constructors:

\begin{tabular}
{|c|c|}
\hline
 Braid groups & Thompson group  \\
 \hline
  Two versions: & Two versions\\
  Unoriented (Plat closure). & Unoriented (all of $F_3$).\\
  Oriented (Trace closure). &  \makecell{Oriented (The subgroup $\overset{\rightarrow}{F_3} <F_3$ \\of definition \ref{oriented} .)}\\
  \hline
  All knots and links  as closure: & All knots and links as $L(g)$:  \\
  Alexander theorem. & ``Alexander theorem''. \ref{alexander}\\
  \hline
 \makecell{Braid index, plat index\\ (=bridge number).} &\makecell{ $\overset{\rightarrow}{F_3}$ index, $F_3$ index,\\ $\overset{\rightarrow}{F_2}$ index, $F_2$ index. } \\
  \hline
 \makecell{ Markov, Birman theorems,\\
 conjugation, double cosets,\\ stabilisation }& \large{??,??,??} \\
  \hline
  Group law: & Group law:\\
  \makecell{Many cobordisms applied to\\ connected sum.} & Connected sum directly.\\
  \makecell{No changes but isotopy \\after these cobordisms.} & \makecell{More cobordisms needed \\after conneced sum.}\\
  \hline 
  \large{??} & $L(gh)=L(g)\#L(h)  \iff |gh|=|g|+|h|$\\ 
  \hline
  \makecell{ All $B_n$'s needed to get \\all knots and links.} & \makecell{One finitely presented group \\gives all knots and links.}\\
\hline
Contains free groups. & Doesn't contain free groups. \\
\hline
Non-amenable & \large{??}\\
\hline
\makecell{Annular version available\\
See \cite{gl}}. & Annular version available.\\
\hline 
\end{tabular} 

\5
One could argue that it is the inductive limit $B_\infty$ of the braid groups that is the correct analogue of a Thompson group, but
even then there is a strong contrast in that $B_\infty$ is not finitely generated. Also the annular version of $B_\infty$ is 
not clear.

\section {Questions.}

We give a list of questions which arise naturally in this work. Some are probably very easy to answer.
\begin{enumerate}
\item "Markov theorem."  

Our theorems concerning the realisation of all links from Thompson group elements are
analogous to Alexander's theorem (\cite{alexander}) which asserts that any (oriented) link may be obtained by closing a braid.
Markov's theorem answers the question of exactly when two braids give the same closure, in terms of simple
changes on the braid group elements. It should be possible to give such a theorem for the Thompson groups $F_3$
and $F_2$. It is easy enough to get moves on group elements that preserve the link, but proving sufficiency of these moves 
has not yet been achieved.
\item A detail about oriented links. 

Theorem \ref{alexander} is very precise-one obtains links without any ambiguity up to distant unlinks (or powers
of $\delta$. But the oriented version, as proved in \cite{jo4}, produces in general links that differ from the desired one by distant
unlinks. So is it true that the Alexander-type theorem for oriented links from a subgroup of $F_2$ or $F_3$ is
true on the nose?
\item A detail about regular isotopy. 

Do all \emph{regular} isotopy classes of link diagrams arise as $L(g)$? (See item \ref{skein} of section  \ref{knots}.)

\item Other Thompson groups.

It is possible to represent links as plane  projections with singularities higher than double points (see, e.g. 
\cite{adams}), e.g. triple and
quadruple points. Such projections naturally arise if one considers the Thompson groups $F_{2k-2}$ as coming from
the category of forests $\mathcal F_{2k-1}$.

\item Proof of theorem \ref{alexander}.

Find a proof more adapted to $F_3$, making the different sign configurations a virtue rather than a vice.
\item Thompson index.

The $F_k$ index of a link $L$ is the smallest number of vertices of a tree such that $L$ is represented 
as the vacuum expectation value of an element of $F_k$ given by a pair of trees with $n$ leaves.
Given that the number of trees with $n$ leaves is finite and the identification of links is algorithmically 
solvable, this is a \emph{finite problem } for a given $L$.

Problem: calculate the $F_3$ and $F_2$ indices of the Borromean rings. (The diagram just before definition
\ref{linkfromf3} shows that the $F_3$ index of the trefoil is 3, its $F_2$ index is more than 3 as can be seen by
enumerating all the 25 pairs of binary planar rooted trees with 3 vertices each.)
\item Irreducibility.

When are the unitary representations of the Thompson groups coming from unitary TQFT's irreducible?
(On the closed $F_n$-linear span of the vacuum.) Some progress was made on this in \cite{joneswysiwyg} where a  family of unitary 
representations using the construction of section \ref{construction} for a TQFT with a slightly different functor $\Phi$
were shown to be irreducible. 
\item Flat connections.

TQFT braid group representations are known to come from monodromy of flat connections on a classifying space 
(KZ connection). 
Can we exhibit the representations of this paper, or at least some of them, as coming from flat connections on 
Belk's (or some other) classifying space \cite{belk}?

\end{enumerate}

\end{document}